\documentclass[10pt]{amsart}
\usepackage{color}
\usepackage{amssymb}
\usepackage{epsfig}
\usepackage{url}
\usepackage{setspace}
\usepackage{algorithm}
\usepackage{algorithmicx}
\usepackage{algpseudocode}
\usepackage{pdflscape}
\usepackage{mathrsfs}

\theoremstyle{plain}

\newtheorem{thm}{Theorem}[section]
\newtheorem{cor}[thm]{Corollary}
\newtheorem{lem}[thm]{Lemma}

\newtheorem{rem}[thm]{Remark}
\newtheorem{prob}{Problem}

\newtheorem{conj}[prob]{Conjecture}
\newtheorem{exam}[thm]{Example}

\def\cal{\mathcal}
\def\bbb{\mathbb}
\def\op{\operatorname}
\renewcommand{\phi}{\varphi}
\newcommand{\R}{\bbb{R}}
\newcommand{\N}{\bbb{N}}
\newcommand{\Z}{\bbb{Z}}
\newcommand{\Q}{\bbb{Q}}

\begin{document}
\title[Hilbert cubes of dimension 3 in the set of squares]{There are infinitely many Hilbert cubes of dimension 3 in the set of squares}

\author{Andrew Bremner, Christian Elsholtz and Maciej Ulas}

\thanks{
Research of C.E. was supported by a joint FWF-ANR project ArithRand (I 4945-N and ANR-20-CE91-0006).\\
Research of M.U. was supported by a grant of the National Science Centre (NCN), Poland, no. UMO-2019/34/E/ST1/00094}

\begin{abstract}
For an integer tuple $(a_0,a_1,\ldots,a_d)$ with
$a_i\neq0$ for $i=1,\ldots,d$, we define the associated Hilbert cube by
\[
H(a_0;a_1,\ldots,a_d)
=
\left\{
a_0+\sum_{i=1}^d \varepsilon_i a_i:
\varepsilon_i\in\{0,1\}
\right\}.
\]
We call $(a_0,a_1,\ldots,a_d)$ a $d$-dimensional Hilbert-cube presentation. 

We prove that there are $\gg N^{1/4}$ reduced Hilbert-cube presentations
$(a_0,a_1,a_2,a_3)$ with $a_0,a_1,a_2,a_3\in[0,N]$ such that
\[
H(a_0;a_1,a_2,a_3)\subset \mathcal{S},
\]
where $\mathcal{S}=\{m^2:\;m\in\Z\}$. Moreover, we prove that for each $i, j\in\{0, 1, 2, 3\}$ with $i<j$, the set
$$
\left\{\frac{a_{i}}{a_{j}}:\;H(a_{0}; a_{1}, a_{2}, a_{3})\subset \cal{S} \;\text{and}\;a_{i}>0\;\text{for}\;i=0, 1, 2, 3\;\right\}
$$
is dense in the set of positive real numbers (in the Euclidean topology).
\end{abstract}
\maketitle
\section{Introduction}\label{sec1}

Throughout the paper, we write $\N=\{1,2,3,\ldots\}$, $\N_0=\N\cup\{0\}$, and
\[
\mathcal S=\{n^2:n\in\mathbb N_0\}
\]
for the set of squares. Thus, in particular, $0\in\mathcal S$.

For $d\geq 1$, let
\[
(a_0,a_1,\ldots,a_d),
\qquad a_i\neq 0 \quad (i=1,\ldots,d),
\]
be a tuple of integers. We define the associated Hilbert cube by
\[
H(a_0;a_1,\ldots,a_d)
=
a_0+\{0,a_1\}+\cdots+\{0,a_d\}
=
\left\{
a_0+\sum_{i=1}^d\varepsilon_i a_i:
\varepsilon_i\in\{0,1\}
\right\}.
\]
We call the tuple $(a_0,a_1,\ldots,a_d)$ a
$d$-dimensional Hilbert-cube presentation.

The term dimension refers to the presentation and need not be an invariant of the associated Hilbert cube, as the example $H(0; 1, 2)=H(0; 1, 1, 1)$ shows. Moreover, we call the Hilbert-cube presentation
$(a_0, a_1,\ldots,a_d)$ {\it reduced} if
\[
0<a_1\leq a_2\leq \cdots\leq a_d\quad\text{and}\quad \gcd(a_0,a_1,\ldots,a_d)=1.
\]

These cubes were introduced by Hilbert \cite{Hilbert:1892} in connection with an irreducibility question on polynomials.
These cubes occur naturally, for example in Szemer\'{e}di's proof of
Szemer\'{e}di's theorem \cite{Szemeredi:1969}, for arithmetic progressions of length $4$ and in Gowers's norm.
They are also well-studied objects in combinatorics; see e.g.~\cite{Conlon-Fox-Sudakov:2014, Hegyvari:1999, Solymosi:2007}.

For a given set $\mathscr{S}$ of positive integers, it is natural to ask
whether infinite Hilbert cubes can exist in $\mathscr{S}$ and, if not, what
the maximal possible dimension of a Hilbert-cube presentation whose associated cube is contained in $\mathscr{S}$ is. It is easy to see that an infinite Hilbert cube does not exist in the squares, as for $i \geq 1$ each $a_i$ would be the difference between squares infinitely often, 
(say $a_i=(a_0+a_i+a_j)-(a_0+a_j)$). But as
differences between consecutive squares are increasing, a fixed difference can be the difference between squares only finitely often.

Before stating our main results, we briefly survey the current state
of knowledge on Hilbert cubes in the squares and related problems.
We return to several related open questions in the final section.

Let us assume that all $a_1, \ldots , a_d$ are non-zero.
Brown, Erd\H{o}s and Freedman \cite{BrownandErdosandFreedman:1990}
asked whether the maximal dimension of a Hilbert-cube presentation whose associated cube is contained in the set $\cal{S}$ of integer squares is \emph{absolutely bounded} or not. 
Hegy\'{v}ari and S\'{a}rk\"{o}zy \cite{HegyvariandSarkozy:1999}
 proved for the set of squares in a finite interval $[1,N]$
that the maximal value of $d$ is bounded by $d=O((\log N)^{1/3})$.
This was improved by Dietmann and Elsholtz \cite{Dietmann-Elsholtz:2012, Dietmann-Elsholtz:2015} first to $d=O((\log \log N)^2)$, and later to $O(\log \log N)$.

Cilleruelo and Granville
\cite{CillerueloandGranville:2007} observed that
the Bombieri-Lang conjecture implies that $d$ is absolutely bounded.
In fact, $x^2+a_1, x^2+a_2$ and $x^2+a_1+a_2$ must be squares for the many distinct values of $x$ in $a_0+\{0,a_3\}+\cdots+\{0,a_d\}$, while the Bombieri-Lang conjecture implies that the curve
$y^2=(x^2+a_1)(x^2+a_2)(x^2+a_1+a_2)$ only has an absolutely bounded number of integral solutions $(x,y)$;
see \cite{CaporasoHarrisandMazur, ShkredovandSolymosi} for details.

Some related problems are as follows:
A well-known open problem by Rudin asks how many squares 
can be in an arithmetic progression $an+b, 1\leq n \leq N$. Rudin conjectured that the progression $24n+1$ has the largest number of squares. For large modulus this is a very difficult question.
Bombieri, Granville and Pintz \cite{BombieriGranvilleandPintz},  and later Bombieri and Zannier \cite{BombieriandZannier} proved that the  maximal number of squares in an arithmetic progression of length $N$ is bounded by $O(N^{2/3}(\log N)^A)$ and $O((\log N)^{3/5}(\log N)^{A'})$, respectively, where $A,A'$ are constants.

One may ask about sumsets $A+B$ in the set of squares. The Hilbert cube associated with a high-dimensional presentation 
can be naturally decomposed into two such summands, in many distinct ways.
Erd\H{o}s and Moser asked whether
there are arbitrarily large sets
such that $a_i+a_j$ are always squares, where $a_i, a_j \in A$ are distinct. 
Rivat, S\'{a}rk\'{o}zy and Stewart \cite{RivatSarkozyandStewart} and 
Gyarmati \cite{Gyarmati:2001} studied the size of sumsets in the squares $\cal{A}+\cal{B}\subset \cal{S}\cap [1,N]$.
For example, the following bound holds (see \cite{Gyarmati:2001}):  
$\min (|\cal{A}|, |\cal{B}|)\leq 8 \log N$.
Related results were obtained by the second author and Wurzinger \cite{Elsholtz-Wurzinger}, for example:
If $|\cal A|\geq \log \log N$, then $|\cal{B}| =O(\log N \log \log N)$.
The second author, Ruzsa and Wurzinger \cite{Elsholtz-Ruzsa-Wurzinger} proved:
If $\cal{A}+\cal{B}+\cal{C}\subset \cal{S}\cap [1,N]$, then
$\min (|\cal{A}|, |\cal{B}|, |\cal{C}|)=O(( \log N)^{4/5})$.

When $|A|=3$, the situation can be described using elliptic curves;
see \cite{DujellaandElsholtz:2013}. In fact, for every sufficiently
large $N$ there exists a set $A$ of three positive integers and a
set $B\subset[0,N]$ with
\[
|B|\gg(\log N)^{15/17}
\]
such that every element of $A+B$ is a square. Moreover, there
exists a set $A$ of three integers and a set $B\subset[0,N]$ with
\[
|B|\gg(\log N)^{9/11}
\]
such that every element of $A$, $B$, and $A+B$ is a square.
When $|A|\geq 4$, the Bombieri--Lang conjecture (or the uniformity
conjecture) implies an absolute bound on $|B|$; see
\cite[Section~4.3]{AlonandAngelandBenjaminiandLubetzky:2012}, \cite{ShkredovandSolymosi}.

For comparison, it is known that the set of squares does not contain any arithmetic progression of length $4$ \cite[Chapter 4, Theorem 3, pp. 21–22.]{Mordell}.

As it turns out, one can find some small Hilbert cubes by hand, and many more by a computer.
Two  easy examples are
\[1+\{0,15\}+\{0,48\}=\{1^2,4^2,7^2,8^2\}\]
and
\[1 + \{0, 528\} + \{0, 840\}+ \{0, 840\} = \{1^2, 23^2, 29^2, 29^2, 37^2, 37^2, 41^2, 47^2\}.\]
Note that in the last example $29^2$ and $37^2$ occur in two different ways, since the set of squares contains \emph{two} arithmetic progressions of length $3$ with gap 840.
According to our computations, the smallest three-dimensional Hilbert-cube presentation with all entries positive and distinct is 
$$
10^2+\{0, 2400\}+\{0, 4389\}+\{0, 8736\}=\{10^2, 50^2, 67^2, 83^2, 94^2, 106^2, 115^2, 125^2\}.
$$
Here, by smallest, we mean the Hilbert-cube presentation $(a_{0}, a_{1}, a_{2}, a_{3})$ with the smallest possible sum $a_{0}+a_{1}+a_{2}+a_{3}$.

We call these base elements $a_i$ generators and will concentrate on positive generators $a_i$. 
Otherwise one could have closely related Hilbert-cube presentations such as
\begin{alignat*}{3}
2209 + \{0, -528\} + \{0, -840\}+ \{0, -840\} &=&
841 + \{0, 528\} + \{0, -840\}+ \{0, 840\}\\ 
&=&
529 + \{0, -528\} + \{0, 840\}+ \{0, 840\}\\
&=&\{1^2, 23^2, 29^2, 29^2, 37^2, 37^2, 41^2, 47^2\}.
\end{alignat*}

It might seem that a Hilbert cube associated with a $d$-dimensional presentation with many overlapping expressions, for example $a_0+a_1+a_2+a_3=a_0+a_4$, might be easier to find, since there are fewer square conditions. A natural question is: how many distinct elements must the Hilbert cube associated with a $d$-dimensional presentation and contained in the squares necessarily have?
Dietmann and the second author \cite[Lemma 1.4]{Dietmann-Elsholtz:2015} gave an exponential lower bound:
$|H(a_0;a_1, \ldots, a_d)|\geq \lceil 2 (\frac{4}{3})^{d-1}\rceil$.
That bound was intended for the asymptotic exponential growth.
For small $d$ it can be adapted as follows:

\begin{lem}\label{lem1}
Let $k \geq 3$ be a positive integer, and let $\mathscr{S}$ denote a set of integers without an arithmetic progression of length $k$. 
Moreover, let $H(a_0; a_1, \ldots , a_d) \subset \mathscr{S}$. Let $d=d_1+d_2$. Then
\[|H| \geq \left\lceil C_{d_1}  \left(\frac{k}{k-1}\right)^{d_2}\right\rceil, \]
where $C_{d_1}$ is the minimal cardinality of a Hilbert cube contained in $\mathscr{S}$ and associated with a $d_1$-dimensional presentation.
\end{lem}

\begin{proof}
Let $A$ be a partial Hilbert cube and let $a$ be the next generator.
Then
\[
A\cup(A+a)\subset \mathscr{S}.
\]
Decompose $A\cup(A+a)$ into maximal chains with common difference
$a$. Since $\mathscr{S}$ contains no arithmetic progression of length $k$,
each such chain has at most $k-1$ elements. The last element of each
chain does not belong to $A$, and hence
\[
|A|\leq \frac{k-1}{k}|A\cup(A+a)|.
\]
Therefore
\[
|A\cup(A+a)|\geq\frac{k}{k-1}|A|.
\]
Applying this inequality successively to the last $d_2$ generators
gives
\[
|H|\geq
C_{d_1}\left(\frac{k}{k-1}\right)^{d_2},
\]
and the result follows after taking the ceiling.
\end{proof}

It is well known that for sets without progressions of length $k= 3$ one even has that $|H| = 2^d$. For the set of squares we have $k=4$.

For $d=1$, the lower bound in Lemma \ref{lem1} gives the correct value $C_1=2$. When $d=2$ this gives the correct value $C_2=\lceil\frac{8}{3}\rceil=3$. 
When $d=3$ the lower bound gives $4$, while $6$ is the actual minimum.

Here are several quite natural open problems.

\begin{prob}\label{prob:d4}
Does a four-dimensional Hilbert-cube presentation exist whose associated cube is contained in the set of squares?
Do infinitely many such presentations exist?
\end{prob}

\begin{prob}\label{prob:euler-brick}
Does a three-dimensional Hilbert-cube presentation with $a_0=0$ whose associated cube is contained in the set of squares exist?
This is the open problem of a perfect Euler brick (or cuboid).
\end{prob}

\begin{prob}\label{prob:unbounded-d}
Can the dimension $d$ of a Hilbert-cube presentation whose associated cube is contained in the set of integer squares be arbitrarily large?
\end{prob}
Solymosi \cite{Solymosi:2007} conjectured that the answer is no.

\begin{prob}
If Problem \ref{prob:unbounded-d} has a negative answer,
what is the largest value of $d$ which occurs at least once,
and what is the largest value of $d$ which occurs for infinitely many Hilbert-cube presentations whose associated cubes are contained in $\cal{S}$?
\end{prob}

To the best of our knowledge, no Hilbert-cube presentation of dimension $d\geq4$ has been found whose associated cube is contained in the set of squares.
Even for $d=2$ and $d=3$, no systematic study is available.
For dimension $d=2$, it is not difficult to write down solutions:
A Hilbert cube $H(a_0;a_1,a_2)=\{a_0, a_0+a_1, a_0+a_2, a_0+a_1+a_2\}$ in the set of squares can be described by
$a_0=x^2, a_1=y^2-x^2, a_2=z^2-x^2$, if and only if
$a_0+a_1+a_2=y^2+z^2-x^2$ is a square.
In other words, every integer with at least two distinct representations as sums of two squares, $y^2+z^2=t^2+x^2$, generates a two-dimensional Hilbert-cube presentation whose associated cube is contained in the squares, and every two-dimensional Hilbert-cube presentation corresponds to a solution of 
$y^2+z^2=t^2+x^2$.
Moreover, if $a_1=a_2$, then $y=z$ and this corresponds to $2y^2=x^2+t^2$, i.e.~three distinct squares in an arithmetic progression. (Note that the case $x=t$ would yield $x=y=t$, and so $a_0=x^2, a_1=y^2-x^2=0, a_2=z^2-x^2=0$).

In this work, we show that infinitely many reduced three-dimensional Hilbert-cube presentations exist, giving an explicit parametric construction. This leads to a good quantitative lower bound for the number of such presentations with entries in $[0,N]$.


The examples above motivate several further questions.

\begin{prob}\label{prob:fixed-a0}
Can there be infinitely many Hilbert-cube presentations with a fixed value of
$a_0$? If not, can one give effective upper bounds for the number
and the sizes of the generators $a_1,a_2,a_3$ in terms of $a_0$?
\end{prob}
\begin{prob}\label{prob:equal-generators}
Are there infinitely many reduced Hilbert-cube presentations of dimension $d\geq3$ with at least two equal generators?
\end{prob}
Compare Theorem \ref{thm:idenical-base}.

\begin{prob}\label{prob:magic}
Does a four-dimensional Hilbert-cube presentation with $a_1=a_2>0$ and $a_3=a_4>0$ whose associated cube is contained in the set of squares exist? Note that if so, then we may construct a $3 \times 3$ magic square of squares, with common row, column, and diagonal sum $3(a_0+a_1+a_3)$:
\begin{center}
        $\Biggl($
\begin{tabular}{ccc}
        $a_0+2 a_1+a_3$ & $a_0$ & $a_0+a_1+2 a_3$ \\
        $a_0+2 a_3$ & $a_0+a_1+a_3$ & $a_0+2 a_1$ \\
        $a_0+a_1$ & $a_0+2 a_1+2 a_3$ & $a_0+a_3$
\end{tabular}
        $\Biggr).$
\end{center}
This special case is a well-known open question: see \cite{Robertson, Bremner:1999, Bremner:2001}. The magic square (see \cite{Bremner:1999})
\begin{center}
        $\Biggl($
\begin{tabular}{ccc}
        $565^2$ & $23^2$ & $222121$ \\
        $289^2$ & $425^2$ & $527^2$ \\
        $373^2$ & $360721$ & $205^2$
\end{tabular}
        $\Biggr)$
\end{center}
leads to $(a_0,a_1,a_2,a_3,a_4)=(23^2, \; 138600, \; 138600, \; 41496, \; 41496)$ with thirteen of the sixteen sums required for the associated cube of a four-dimensional presentation to be contained in the squares being squares.
\end{prob}
\begin{prob}\label{prob:effective}
     Given $a_0,\ldots, a_{j}$, where $j\in\{0,1,2\}$, what are effective search bounds for the remaining parameters $a_{i}$, and how can one algorithmically extend the given parameters to a Hilbert-cube presentation whose associated cube is contained in the squares?
     \end{prob}
Compare Algorithm~1 below.\\

We address some of these problems in this paper and give the following 
results.

\begin{thm}{\label{thm:finite}}
\begin{enumerate}
\item  For any fixed square $a_0$ there are infinitely many two-dimensional Hilbert-cube presentations with this value of $a_0$ whose associated cubes are contained in the set of squares.
\item For fixed $a_1>0$ and fixed $d\geq2$, there are only finitely many $d$-dimensional Hilbert-cube presentations whose associated cubes are contained in the set of squares. The bound is effective:
\begin{enumerate}
\item
For each $a_i$, for $i=2,\ldots, d$, the upper bound $a_i\leq \frac{1}{4}(a_1-1)^2-a_0$ holds.
\item
For each $a_i$, for $i=2,\ldots, d$, there are at most $\tau(a_1)$ choices, where $\tau(n)$ is the number of positive divisors of $n$.
\end{enumerate}
\end{enumerate}
\end{thm}

Let us introduce the set
$$
\cal{H}=\{(a_{0}, a_{1},a_{2},a_{3})\in\N_{0}\times (\Z\setminus\{0\})^{3}:\;H(a_{0}; a_{1}, a_{2}, a_{3})\subset \cal{S}\}
$$
of all three-dimensional Hilbert-cube presentations whose associated cubes are contained in the set of squares. Moreover, for $d\in \N$ we define

$$
H_d(N)
=
\#\left\{
(a_0, a_1,\ldots,a_d):
\begin{array}{l}
H(a_0;a_1,\ldots,a_d)\subset\mathcal S,\\
(a_0, a_1,\ldots,a_d)\;\text{is reduced},\\
0\leq a_i\leq N\quad\text{for }i=0,\ldots,d
\end{array}
\right\}.
$$

\begin{thm}{\label{thm:H2}}
The number $H_2(N)$ of reduced two-dimensional Hilbert-cube presentations satisfies
$$
N\log N \ll H_2(N) \ll N\log N.
$$
\end{thm}

\begin{thm}{\label{thm:H3}}
The number $H_3(N)$ of reduced three-dimensional Hilbert-cube presentations satisfies
$$
H_3(N)\gg N^{1/4}.
$$
\end{thm}

In the next theorem  we give a general upper bound for $H_{d}(N)$. 

\begin{thm}\label{thm:H_d-upper-bound}
For every fixed $d\geq2$ and every $\varepsilon>0$, the number
$H_d(N)$ of reduced $d$-dimensional Hilbert-cube presentations
satisfies
$$
H_d(N)\ll_{d,\varepsilon}N^{1+\varepsilon}.
$$
\end{thm}

\begin{thm}\label{thm:idenical-base}
There exist infinitely many reduced Hilbert-cube presentations in $\mathcal H$ with $a_1=a_2<a_3$. Moreover, the number of such presentations with $a_i \leq N$ is at least $cN^{1/5}$, for some positive constant $c$. The same statement holds for $a_1<a_2=a_3$.

\end{thm}

From the parametric family of three-dimensional Hilbert-cube presentations one can deduce the following:
\begin{cor}{\label{cor-fractions}}
Let $\cal{H}_{+}=\cal{H}\cap \N^{4}$. For each $i, j\in\{0, 1, 2, 3\}$ with $i<j$, the set
$$
H_{i,j}=\left\{\frac{a_{i}}{a_{j}}:\;(a_{0}, a_{1}, a_{2}, a_{3})\in\cal{H}_{+}\right\}\subset \Q_{+}
$$
is dense in $\R_+$ with respect to the Euclidean topology.
\end{cor}

The question about Hilbert cubes in the set of quadratic residues has been studied by
\cite{Alsetri-Shao, Dietmann-Elsholtz:2012,
Dietmann-Elsholtz-Shparlinski,
HegyvariandSarkozy:1999, Schoen}. A recent breakthrough is by Kalmynin \cite{Kalmynin}, who proved that the set of quadratic residues modulo a prime $p$ cannot be written as a sumset $A+B$ with $|A|, |B|>1$.

A closely related question is about sumsets $A+B$ contained in the set of squares. There is some information 
in \cite{AlonandAngelandBenjaminiandLubetzky:2012, DujellaandElsholtz:2013}.

\section{Notation}
We use the Landau and Vinogradov notation: 
for functions $f:\N\rightarrow \N$ and $g:\N \rightarrow \N$
we write $f \ll g, g\gg f$ or $f = O(g)$ if there exists a constant $C>0$
such that $f (n) \leq C g(n)$ for all $n$. 
If the constant $C$ depends on some parameter $\varepsilon$ this is denoted by $O_{\varepsilon}$ or $\ll_{\varepsilon}$.

The number of positive divisors of an integer $n$ is denoted by $\tau(n)$.

Let us also introduce the following notation.

\begin{equation}\label{Hsystem}
\begin{aligned}
p^2&=a_{0},\\
q^2&=a_{0}+a_{1},\\
r^2&=a_{0}+a_{2},\\
s^2&=a_{0}+a_{1}+a_{2},\\
P^2&=a_{0}+a_{3},\\
Q^2&=a_{0}+a_{1}+a_{3},\\
R^2&=a_{0}+a_{2}+a_{3},\\
S^2&=a_{0}+a_{1}+a_{2}+a_{3},
\end{aligned}
\end{equation}
for some integers $p, q,r,s, P,Q,R, S$.

\section{Small solutions}\label{sec2}
As usual, when facing a Diophantine problem, it is natural to search computationally for small integer solutions. Before presenting the results of our initial computations, we describe some natural symmetries of $\cal{H}$ that allow us to normalize the signs and ordering of the generators and thus reduce the search space.  Indeed, each of the following maps preserves $\cal{H}$: 
\begin{align*}
\phi_{1}&:\;\cal{H}\ni (a_{0}, a_{1}, a_{2}, a_{3})\mapsto (a_{0}+a_{1}, -a_{1}, a_{2}, a_{3})\in \cal{H},\\
\phi_{2}&:\;\cal{H}\ni (a_{0}, a_{1}, a_{2}, a_{3})\mapsto (a_{0}+a_{2}+a_{3}, a_{1}, -a_{3}, -a_{2})\in \cal{H},\\
\phi_{\sigma}&:\;\cal{H}\ni (a_{0}, a_{1}, a_{2}, a_{3})\mapsto (a_{0}, a_{\sigma(1)}, a_{\sigma(2)}, a_{\sigma(3)})\in \cal{H},
\end{align*}
where $\sigma$ is an element of the symmetric group $\Sigma_3$ on three elements acting on $\cal{H}$. 

We have $\phi_1^2=\operatorname{id}_{\cal H}$. Conjugating $\phi_1$ by the maps $\phi_\sigma$, with 
$\sigma\in\Sigma_3$, gives the three involutions corresponding to
independent changes of sign of the generators $a_1,a_2,a_3$, together
with the corresponding change of the base point $a_0$. These three
involutions commute and generate a subgroup isomorphic to
$(\mathbb Z/2\mathbb Z)^3$. On the other hand, the maps $\phi_\sigma$
form a subgroup isomorphic to $\Sigma_3$ and permute these three
involutions by conjugation. Moreover, $\phi_2$ is already contained
in the group generated by $\phi_1$ and the maps $\phi_\sigma$.
Consequently, the group generated by the above maps is isomorphic to
$(\mathbb Z/2\mathbb Z)^3\rtimes\Sigma_3$ and hence has order $2^3\cdot 3!=48$.

Moreover, for each positive integer $m$ we have an additional map
$$
\psi_{m}:\;\cal{H}\ni (a_{0}, a_{1}, a_{2}, a_{3})\mapsto  (m^2a_{0}, m^2a_{1}, m^2a_{2}, m^2a_{3})\in\cal{H}.
$$

Consequently, after applying suitable compositions of these maps and their inverses, we may assume that $0<a_1\leq a_2\leq a_3$. After separately dividing by a common square factor when necessary, we obtain a reduced presentation.

We begin with the following simple lemma.

\begin{lem}\label{lem:comp}
Let $m\geq2$ be fixed. Then, for any fixed $a_0\in\mathbb N_0$ and $a_1\in\mathbb N$, there are only finitely many Hilbert-cube presentations $(a_0, a_1,\ldots,a_m)$ whose associated cubes are contained in the set of squares.
\end{lem}
\begin{proof}
We start with the case $m=2$. We have $H(a_{0}; a_{1}, a_{2})\subset \cal{S}$ if and only if
$$
a_{0}=p^2,\quad  a_{0}+a_{1}=q^{2}, \quad  a_{0}+a_{2}=r^2, \quad a_{0}+a_{1}+a_{2}=s^2
$$
for some $p, q, r, s$. Since $a_{0}, a_{1}$ are fixed, $p, q$ are also fixed. Subtracting the third equation from the fourth gives $a_{1}=s^2-r^2=(s-r)(s+r)$. Thus, if $d$ is a divisor of $a_1$, then the corresponding values of $r, s$ are of the form
$$
r=\frac{1}{2}\left(\frac{a_{1}}{d}-d\right),\quad s=\frac{1}{2}\left(\frac{a_{1}}{d}+d\right).
$$
For each admissible pair $(r,s)$, we set $a_2=r^2-a_0$. 
Since the number of possible values of $r$ is bounded by $\tau(a_{1})$, where $\tau(n)$ is the number of divisors of $n$, the number of possible values of 
$a_{2}$ is also bounded by $\tau(a_{1})$. This proves the lemma for $m=2$. By iterating this argument, we see that for any given $a_{0}, a_{1}$ and $m$, the number of Hilbert-cube presentations $(a_{0}, a_{1}, a_{2},\ldots, a_{m})$ whose associated cubes are contained in the set of squares is finite.
\end{proof}

\begin{proof}[Proof of Theorem~\ref{thm:finite}]
We first prove part~(1). Let $a_0=p^2$ with $p>0$. For any
integer $t\geq 2$, put
\[
q=p(2t+1),\qquad
r=p(t(t+1)-1),\qquad
s=p(t(t+1)+1).
\]
Then
\[
q^2-p^2=s^2-r^2=4p^2t(t+1).
\]
Thus, setting
\[
a_1=q^2-p^2,\qquad a_2=r^2-p^2,
\]
we obtain
\[
H(a_0;a_1,a_2)
=
\{p^2,q^2,r^2,s^2\}
\subset\mathcal S.
\]
Since $a_1=4p^2t(t+1)$ takes infinitely many values, this gives
infinitely many such presentations with the prescribed value of
$a_0$.

If $a_0=0$, the same conclusion follows, for example, from
\[
H\bigl(0;(t^2-1)^2,4t^2\bigr)
=
\{0,(t^2-1)^2,(2t)^2,(t^2+1)^2\}
\subset\mathcal S,
\qquad t\geq2.
\]

We now prove part~(2). Fix $a_1>0$. Since
\[
a_0=p^2,\qquad a_0+a_1=q^2,
\]
we have
\[
a_1=q^2-p^2=(q-p)(q+p).
\]
Hence there are only finitely many possible pairs $(p,q)$, and
therefore only finitely many possible values of $a_0$.

Once $a_0$ is fixed, the argument from Lemma~\ref{lem:comp}
applies. For each $i=2,\ldots,d$, write
\[
a_0+a_i=r_i^2,\qquad
a_0+a_1+a_i=s_i^2.
\]
Then
\[
a_1=s_i^2-r_i^2=(s_i-r_i)(s_i+r_i).
\]
Thus $s_i-r_i$ is a divisor of $a_1$, and so there are at most
$\tau(a_1)$ possible values of $a_i$. Moreover,
\[
r_i
=
\frac{1}{2}
\left(\frac{a_1}{s_i-r_i}-(s_i-r_i)\right)
\leq \frac{a_1-1}{2},
\]
and consequently
\[
a_i=r_i^2-a_0
\leq
\frac{(a_1-1)^2}{4}-a_0.
\]
This proves the stated effective bounds and the finiteness claim.
\end{proof}

We now prove the quantitative estimate for $H_2(N)$.

\begin{proof}[Proof of Theorem \ref{thm:H2}]
Let $r_2(n)$ denote the number of representations of $n$ as a sum
of two squares, with signs and order taken into account. By a result of Sierpi\'{n}ski \cite[last line of page 53]{Sierpinski:1907}, 
but also see \cite[Eq.~(3.14)]{BorweinChoi} for a modern exposition,
\[
\sum_{n\leq N} r_2(n)^2= 4N\log N+O(N).
\]

We first prove the upper bound. Every reduced presentation
\[
(a_0, a_1,a_2)
\]
whose associated Hilbert cube is contained in $\mathcal S$ gives nonnegative integers $p,q,r,s$ such that
\[
a_0=p^2,\qquad
a_0+a_1=q^2,\qquad
a_0+a_2=r^2,\qquad
a_0+a_1+a_2=s^2.
\]
Hence
\[
p^2+s^2=q^2+r^2.
\]
If $0\leq a_0,a_1,a_2\leq N$, then
\[
p^2+s^2=2a_0+a_1+a_2\leq4N.
\]
Therefore
\[
H_2(N)
\leq
\sum_{n\leq4N}r_2(n)^2
\leq 16 N\log N +O(N).
\]

For the lower bound, let
\[
R(n)
=
\#\{(x,y)\in\mathbb N^2:\ 0<x<y,\ x^2+y^2=n\}.
\]
Except when $n$ is a square or twice a square, each representation
counted by $R(n)$ gives exactly eight representations counted by
$r_2(n)$, corresponding to the choices of signs and the order of
the two coordinates. Since there are only $O(\sqrt N)$ exceptional
values and
\[
r_2(n)\ll_\varepsilon n^\varepsilon,
\]
their contribution to the second moment is $O(N^{0.51})$. Thus
\[
\sum_{n\leq N}R(n)^2
=
\frac{1}{64}\sum_{n\leq N}r_2(n)^2+O(N)
=
\frac{1}{16}N\log N+O(N).
\]

We now restrict to pairs of representations
\[
n=x_1^2+y_1^2=x_2^2+y_2^2
\]
satisfying
\[
\gcd(x_1,y_1,x_2,y_2)=1.
\]
By M\"{o}bius inversion, the number of ordered pairs of such primitive
representations with $n\leq N$ is
\[
\sum_{m\leq\sqrt N}\mu(m)
\sum_{n\leq N/m^2}R(n)^2,
\]
where $\mu$ is the M\"{o}bius function. Using the preceding asymptotic,
this equals
\[
\frac{N}{16}
\sum_{m\leq\sqrt N}
\frac{\mu(m)}{m^2}
\log\frac{N}{m^2}
+O(N).
\]
Since
\[
\sum_{m=1}^{\infty}\frac{\mu(m)}{m^2}
=
\frac{1}{\zeta(2)},
\]
we obtain
\[
\sum_{m\leq\sqrt N}\mu(m)
\sum_{n\leq N/m^2}R(n)^2
=
\frac{1}{16\zeta(2)}N\log N+O(N).
\]
The diagonal pairs contribute only
\[
\sum_{n\leq N}R(n)=O(N),
\]
since this sum counts lattice points $(x,y)$ satisfying
$x^2+y^2\leq N$. Consequently, there are
\[
\gg N\log N
\]
unordered pairs of distinct primitive representations.

Given such a pair, we may order the two representations so that
\[
n=x_1^2+y_1^2=x_2^2+y_2^2,
\qquad
0<x_1<x_2<y_2<y_1.
\]
Then
\[
H\left(
x_1^2;\
x_2^2-x_1^2,\,
y_2^2-x_1^2
\right)
=
\{x_1^2,x_2^2,y_2^2,y_1^2\}
\subset\mathcal S.
\]
All three parameters are at most $N$. Since $0<x_1<x_2<y_2$, the resulting generators satisfy
$$
0<x_2^2-x_1^2<y_2^2-x_1^2.
$$ 
Moreover, the presentation is
reduced, since
\[
\begin{aligned}
\gcd\left(
x_1^2,\,
x_2^2-x_1^2,\,
y_2^2-x_1^2
\right)
&=
\gcd(x_1^2,x_2^2,y_2^2,y_1^2)\\
&=
\gcd(x_1,x_2,y_2,y_1)^2
=1.
\end{aligned}
\]
Finally, the pair of representations is uniquely determined by the
resulting presentation. Hence
\[
H_2(N)\gg N\log N.
\]
Together with the upper bound, this proves
\[
H_2(N)\asymp N\log N.
\]
\end{proof}

The method used in the proof of Lemma~\ref{lem:comp} also yields
a general upper bound for $H_d(N)$.

\begin{proof}[Proof of Theorem~\ref{thm:H_d-upper-bound}]
There are $O(N)$ possible pairs $(a_0,a_1)$ satisfying the system
$$
a_0=p^2,\qquad a_0+a_1=q^2.
$$
Indeed, we have $p\leq\sqrt N$ and $q\leq\sqrt{2N}$.

By the argument in the proof of Lemma~\ref{lem:comp}, once
$a_0$ and $a_1$ are fixed, each generator $a_i$, $i=2,\ldots,d$,
has at most $\tau(a_1)$ possible values. Hence
$$
H_d(N)
\ll_d
N\max_{n\leq N}\tau(n)^{d-1}.
$$
The standard divisor bound implies that, for every $\varepsilon>0$,
$$
\tau(n)^{d-1}\ll_{d,\varepsilon}n^\varepsilon.
$$
Therefore
$$
H_d(N)\ll_{d,\varepsilon}N^{1+\varepsilon}
$$
and our theorem is proved.
\end{proof}

Although very simple, the method from the proof of Lemma \ref{lem:comp} can be used to write an efficient algorithm for finding all Hilbert-cube presentations $(a_{0}, a_{1}, a_{2}, a_{3})$ whose associated cubes are contained in $\cal{S}$ with $a_{0}, a_{1}\leq N$ for given $N$ in time $O_\varepsilon(N^{1+\varepsilon})$. The algorithm runs as follows. First, for a given $a_{1}$ we compute the set 
$$
D_{small}(a_{1})=\{d\in\N:\;d|a_{1}\;\mbox{and}\; d^2<a_{1}\}=\{d_{1}, d_{2}, \ldots, d_{m}\},
$$
in natural order,
where $m=m(a_{1})=\#D_{small}(a_{1})$. Next, for given $d_{i}, d_{j}\in D_{small}(a_{1}), i<j$, we check whether 
\begin{align*}
&c_{1}(i,j)=\frac{1}{4}\left(\frac{a_{1}}{d_{i}}-d_{i}\right)^2
+\frac{1}{4}\left(\frac{a_{1}}{d_{j}}-d_{j}\right)^2-a_{0}=\square,\\
&c_{2}(i,j)=a_{1}+\frac{1}{4}\left(\frac{a_{1}}{d_{i}}-d_{i}\right)^2
+\frac{1}{4}\left(\frac{a_{1}}{d_{j}}-d_{j}\right)^2-a_{0}=\square.
\end{align*}

If these conditions are satisfied, then we put 
$$
a_{2, i}=\frac{1}{4}\left(\frac{a_{1}}{d_{i}}-d_{i}\right)^2-a_{0},\quad a_{3, j}=\frac{1}{4}\left(\frac{a_{1}}{d_{j}}-d_{j}\right)^2-a_{0}.
$$
Provided that $H(a_0;a_1)\subset\mathcal S$, this yields a Hilbert cube $H(a_{0}; a_{1}, a_{2,i}, a_{3, j})\subset \cal{S}$. 
Note that the algorithm only needs to test two square conditions.

The pseudocode of our algorithm has the following form. We implemented it in PARI/GP \cite{PARI2}. The PARI/GP implementation of Algorithm~1 is available from the authors upon request.

\begin{algorithm}[H]\label{algorithm}
\caption{Finding all reduced Hilbert-cube presentations
$(a_0, a_1,a_2,a_3)$ whose associated cubes are contained in $\mathcal S$ with
$0\leq a_0\leq N$ and $0<a_1\leq N$}
\begin{algorithmic}[1]
\Procedure{HilbertCubes}{$N$}
  \State $\mathcal L\gets\varnothing$

  \For{$p=0$ \textbf{to} $\lfloor\sqrt N\rfloor$}
    \State $a_0\gets p^2$

    \For{$q=p+1$ \textbf{to}
      $\lfloor\sqrt{a_0+N}\rfloor$}
      \State $a_1\gets q^2-a_0$

      \State
      $D\gets
      \{d\in\mathbb N:
        d\mid a_1,\;
        d^2<a_1,\;
        a_1/d\equiv d\pmod 2
      \}$

      \State $\mathcal A\gets\varnothing$

      \ForAll{$d\in D$}
        \State
        $r\gets
        \dfrac{1}{2}
        \left(\dfrac{a_1}{d}-d\right)$
        \State $a\gets r^2-a_0$

        \If{$a>0$}
          \State $\mathcal A\gets\mathcal A\cup\{a\}$
        \EndIf
      \EndFor

      \ForAll{$(b,c)\in\mathcal A^2$ with $b\leq c$}

        \If{$a_0+b+c$ is not a square}
          \State \textbf{continue}
        \EndIf

        \If{$a_0+a_1+b+c$ is not a square}
          \State \textbf{continue}
        \EndIf

        \State
        $(v_1,v_2,v_3)
        \gets\operatorname{sort}(a_1,b,c)$

        \If{$\gcd(a_0,v_1,v_2,v_3)=1$}
          \State
          $\mathcal L
          \gets
          \mathcal L
          \cup
          \{(a_0, v_1,v_2,v_3)\}$
        \EndIf

      \EndFor
    \EndFor
  \EndFor

  \State \Return $\mathcal L$
\EndProcedure
\end{algorithmic}
\end{algorithm}

\begin{exam}\label{exam1}
{\rm To see the described algorithm in action let us consider the following examples. 
Take $a_{0}=1, a_{1}=8099=90^2-1$, i.e., $H(a_{0}; a_{1})\subset \cal{S}$. We have 
$$
D_{small}(a_{1})=\{1, 7, 13, 89\}=\{d_{1}, d_{2}, d_{3}, d_{4}\}.
$$
The corresponding values of $c_{1}(i, j), c_{2}(i, j)$ are presented in Tables~\ref{table1} and~\ref{table2}. 
\begin{table}[htbp]
\centering
\begin{tabular}{|l|llll|}
\hline
 $i\backslash j$ & 1       & 2        & 3        & 4\\
 \hline
   1           &32788801 & 16725025 & 16487425 & $4049^2$ \\
    2          &&661249 & 423649 & $575^2$  \\
     3         &&&186049 & $305^2$  \\
      4        &&&&1   \\
      \hline
\end{tabular}
\caption{Values of $c_{1}(i, j)$, $1\leq i<j\leq 4$ for $a_{0}=1, a_{1}=8099$.}
\label{table1}
\end{table}

\begin{table}[htbp]
\centering
\begin{tabular}{|l|llll|}
\hline
 $i\backslash j$ & 1       & 2        & 3        & 4\\
 \hline
1 &32796897 & 16733121 & 16495521 & 16402497 \\
2 &&669345 & 431745 & 338721  \\
3 &&&194145 & 101121  \\
4 &&&&8097 \\
     \hline
\end{tabular}
\caption{Values of $c_{2}(i, j)$, $1\leq i<j\leq 4$ for $a_{0}=1, a_{1}=8099$}
\label{table2}
\end{table}

We note that $c_{1}(i, j)$ is a square if and only if $j=4$. However, the value $c_{2}(i, 4)$ is not a square for any $i=1, 2, 3,4 $. Consequently, there cannot be any Hilbert-cube presentation of dimension $d\geq3$ whose associated cube is contained in the set of squares with $a_0=1$ and $a_1=8099$.

To see another example, let us take $a_{0}=4$ and $a_{1}=3360=58^2-4$. We then have
\begin{align*}
D_{small}(a_{1})=&\{1, 2, 3, 4, 5, 6, 7, 8, 10, 12, 14, 15, 16, 20,\\
                 & \;\;21, 24, 28, 30, 32, 35, 40, 42, 48, 56\}.
\end{align*}
Performing the necessary calculations, we find that 
$$
c_{1}(13,16)=113^2,\quad c_{2}(13,16)=127^2
$$ 
(related to the divisors $d_{13}=16, d_{16}=24$ of $a_{1}$) and we obtain the Hilbert-cube presentation $(4, 3360,9405,3360)$, whose associated cube is contained in $\cal S$. Interestingly, there is another solution given by 
$$
c_{1}(16,17)=74^2,\quad c_{2}(16,17)=94^2
$$ 
(related to the divisors $d_{16}=24, d_{17}=28$ of $a_{1}$) and one obtains the Hilbert-cube presentation $(4, 2112,3360,3360)$, whose associated cube is contained in $\cal S$. This presentation is not reduced, since $\gcd(4,2112,3360,3360)=4$.}
\end{exam}

The above result allows us to find all reduced three-dimensional Hilbert-cube presentations with $a_{0}, a_{1}<10^8$. To do this, we used the described algorithm with $a_{0}=p^2, a_{1}=q^2-p^2$ with $p\leq 10^4,  p<q\leq \sqrt{p^2+10^{8}}$. Under these assumptions we found 4644 reduced three-dimensional Hilbert-cube presentations whose associated cubes are contained in the set of squares (the computations took about one day). None of these presentations can be extended to a four-dimensional Hilbert-cube presentation. In Table \ref{smallcubes} we present all reduced three-dimensional Hilbert-cube presentations $(a_{0}, a_{1}, a_{2}, a_{3})$ whose associated cubes are contained in $\cal S$ with $\op{max}\{a_{0}, a_{1}, a_{2}, a_{3}\}\leq 2^{19}.$

In this context it is natural to consider the quantity
$$
C_{3}(N)=\#\{(a_{0}, a_{1})\in [0,N]^2:\;a_{0}<a_{1}, H(a_{0}; a_{1}, a_{2}, a_{3})\subset\cal{S}\;\mbox{for some}\;a_{2}, a_{3}\in\N\}.
$$
Table~\ref{C3H3} gives the values of $H_3(m)$ and $C_3(m)$ for $m=2^n$ and $n=15,\ldots,26$. 
\begin{table}[htbp]
\centering
\begin{tabular}{|c|c|c|c|c|c|c|c|c|c|c|c|c|}
\hline
 $n$         & 15 & 16 & 17 & 18 & 19 & 20 & 21 & 22 & 23 & 24 & 25 & 26 \\
 \hline 
$H_{3}(2^{n})$  & 8 & 13 & 28 & 39 & 65 & 99 & 147 & 239 & 363 & 529 & 792 & 1173 \\
 \hline
 $C_{3}(2^{n})$   &  27 & 51 & 91 & 136 & 228 & 349 & 541 & 852 & 1278 & 1851 & 2730 & 3887\\
 \hline
\end{tabular}
\caption{Values of $H_{3}(2^n)$ and $C_{3}(2^n)$ for $n\in\{15,\ldots, 26\}$.}
\label{C3H3}
\end{table}

\begin{figure}[!ht]
    \centering
    \includegraphics[width=0.9\linewidth]{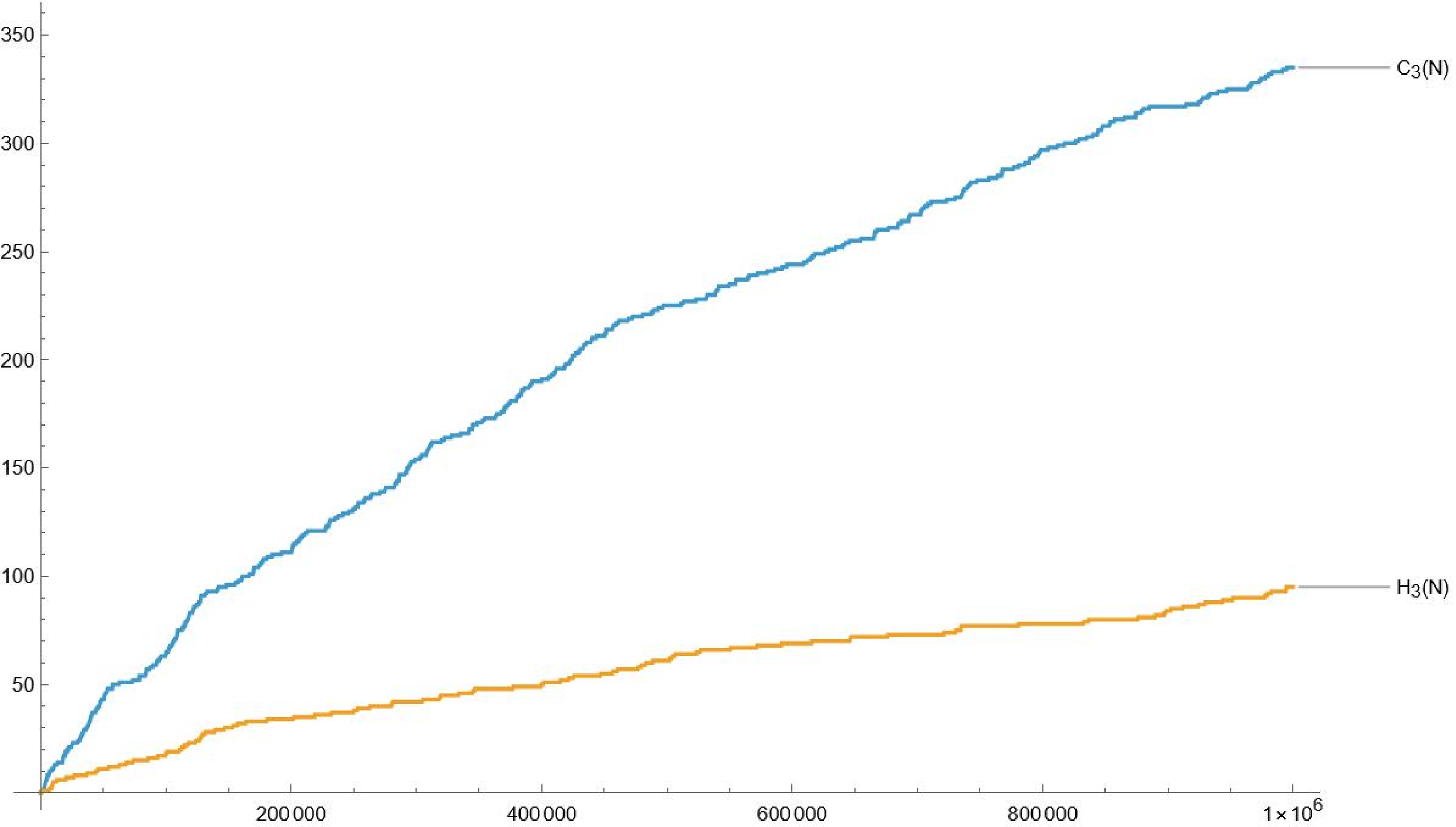}
    \caption{Plot of the functions $C_{3}(n)$ (blue) and $H_{3}(n)$ (yellow) for $n\leq 10^6$. In particular, $C_{3}(N)\geq H_{3}(N)$.}
    \label{fig:Pic3}
\end{figure}

\begin{table}[htbp]
\centering
\begin{tabular}{|c|c|c|c||c|c|c|c|}
\hline
$a_0$ & $a_1$ & $a_2$ & $a_3$ & $a_0$ & $a_1$ & $a_2$ & $a_3$ \\
\hline
 1 & 528 & 840 & 840 & $73^2$ & 31920 & 31920 & 123552 \\
 $2^2$ & 3360 & 3360 & 9405 & $93^2$ & 200200 & 407376 & 488376 \\
 $2^2$ & 4485 & 7392 & 20160 & $103^2$ & 43680 & 43680 & 100947 \\
 $2^2$ & 17952 & 26565 & 131040 & $103^2$ & 105672 & 207480 & 333960 \\
 $2^2$ & 100485 & 209760 & 376992 & $107^2$ & 18480 & 28152 & 127680 \\
 $5^2$ & 32736 & 46200 & 54264 & $107^2$ & 57195 & 364320 & 446880 \\
 $5^2$ & 78936 & 218064 & 319200 & $107^2$ & 102120 & 247632 & 414960 \\
 $7^2$ & 5280 & 5280 & 9555 & $109^2$ & 19803 & 36960 & 163680 \\
 $7^2$ & 40755 & 108192 & 252960 & $120^2$ & 106704 & 426496 & 482625 \\
 $7^2$ & 127400 & 180576 & 180576 & $125^2$ & 22011 & 84864 & 117600 \\
 $10^2$ & 2400 & 4389 & 8736 & $127^2$ & 33600 & 46371 & 280896 \\
 $13^2$ & 10440 & 15960 & 62832 & $129^2$ & 17955 & 40480 & 157248 \\
 $17^2$ & 45936 & 134400 & 249711 & $129^2$ & 177840 & 366520 & 506088 \\
 $23^2$ & 41496 & 138600 & 138600 & $130^2$ & 94656 & 181125 & 421344 \\
 $31^2$ & 10920 & 10920 & 26928 & $139^2$ & 115368 & 182280 & 318240 \\
 $31^2$ & 38640 & 57120 & 99528 & $141^2$ & 37240 & 201960 & 201960 \\
 $34^2$ & 3885 & 73920 & 73920 & $146^2$ & 127680 & 127680 & 231693 \\
 $46^2$ & 3360 & 3360 & 7293 & $151^2$ & 48488 & 128520 & 128520 \\
 $46^2$ & 5280 & 9765 & 12768 & $158^2$ & 87261 & 184800 & 479136 \\
 $47^2$ & 122400 & 122400 & 401016 & $162^2$ & 155232 & 246240 & 279565 \\
 $53^2$ & 21216 & 28875 & 153216 & $167^2$ & 201552 & 425040 & 425040 \\
 $58^2$ & 44160 & 94605 & 110880 & $197^2$ & 31416 & 73416 & 85800 \\
 $59^2$ & 25080 & 65688 & 93240 & $201^2$ & 53235 & 112480 & 399168 \\
 $62^2$ & 6765 & 43680 & 43680 & $201^2$ & 170280 & 294840 & 476560 \\
 $62^2$ & 7392 & 16320 & 45885 & $213^2$ & 19656 & 19656 & 68200 \\
 $62^2$ & 8925 & 23712 & 36960 & $226^2$ & 322245 & 350880 & 523488 \\
 $64^2$ & 62985 & 270480 & 346368 & $235^2$ & 83904 & 96096 & 112875 \\
 $67^2$ & 47040 & 86112 & 126555 & $238^2$ & 47040 & 74400 & 460317 \\
 $67^2$ & 52632 & 64680 & 503880 & $282^2$ & 78880 & 125685 & 218592 \\
 $67^2$ & 89760 & 120120 & 146832 & $337^2$ & 84456 & 231000 & 456456 \\
 $69^2$ & 51408 & 196840 & 344520 & $503^2$ & 127680 & 127680 & 262515 \\
 $70^2$ & 36309 & 79200 & 114816 & $595^2$ & 116571 & 501600 & 501600 \\
 $73^2$ & 29640 & 224112 & 304920 &  &  &  &  \\
 \hline
\end{tabular}
\caption{Complete table of reduced Hilbert-cube presentations $(a_{0}, a_{1}, a_{2}, a_{3})$ whose associated cubes are contained in $\cal{S}$ with $\op{max}\{a_{0}, a_{1}, a_{2}, a_{3}\}\leq 2^{19}.$ }  \label{smallcubes}
\end{table}
\bigskip
Our computations strongly suggest that there are infinitely many reduced three-dimensional Hilbert-cube presentations whose associated cubes are contained in the set of squares. Moreover, many of the solutions satisfy one of the additional conditions $a_{1}=a_{2}$ or $a_{2}=a_{3}$. This also suggests that there are infinitely many elements of $\cal{H}$ with only three different entries. 

As we will see in the next section, these expectations are true.

\section{Proof of Theorem~\ref{thm:H3} and
Theorem ~\ref{thm:idenical-base}. }

\subsection{Construction of a parametric family of three-dimensional Hilbert-cube presentations}
Let $a_{0}, a_{1}, a_{2}, a_{3}\in\Z$ and suppose that
$$
H(a_{0}; a_{1}, a_{2}, a_{3})=a_{0}+\{0, a_{1}\}+\{0, a_{2}\}+\{0, a_{3}\}\subset \cal{S}.
$$
We solve the system~(\ref{Hsystem}) by treating its equations successively.

The first three equations can be easily solved by taking
$$
a_{0}=p^2, \quad a_{1}=q^2-p^2,\quad a_{2}=r^2-p^2.
$$
The fourth equation is equivalent to the equation $q^2-p^2=s^2-r^2$, and thus
$$
p=\frac{1}{2} (t z-x y),\quad q=\frac{1}{2} (t z+x y),\quad r=\frac{1}{2} (y z-t x),\quad s=\frac{1}{2} (t x+y z).
$$
We solve the fifth equation with respect to $a_{3}$ and obtain
$$
a_{3}=P^2-p^2=\frac{1}{4} \left(4 P^2-t^2 z^2+2 t x y z-x^2 y^2\right).
$$
The sixth equation takes the form
$$
P^2 + t x y z = Q^2.
$$
Solving for $t$, we obtain $t=\frac{Q^2-P^2}{x y z}$. Substituting this value and the expressions for $p,q,r,s$, we obtain the corresponding values of $a_0$, $a_1$, $a_2$, $a_3$ as follows:
\begin{equation}\label{a0a1a2a3}
\begin{aligned}
	a_0 & = (P^2-Q^2+x^2y^2)^2/(4 x^2 y^2), \\
	a_1 & = Q^2-P^2, \\
	a_2 & = (x^2-z^2)(P^2-Q^2+x y^2 z)(P^2-Q^2-x y^2 z)/(4 x^2 y^2 z^2), \\
	a_3 & = -(P+Q-x y)(P+Q+x y)(P-Q+x y)(P-Q-x y)/(4 x^2y^2).
\end{aligned}
\end{equation}
Replacing 
$(R,S)$ by $(R,S)/(2 x y z)$, the remaining two equations take the form
\begin{equation}\label{sysRS}
\begin{aligned}
	R^2 & = x^2(P^2-Q^2)^2 + x^2y^4z^4 - ((P^2-Q^2)^2-x^2 y^2(2P-x y)(2P+x y
)) z^2, \\
	S^2 & = x^2(P^2-Q^2)^2 + x^2y^4z^4 - ((P^2-Q^2)^2-x^2 y^2(2Q-x y)(2Q+x y
)) z^2.
\end{aligned}
\end{equation}
Set $(P,Q) = (\frac{1}{2}(u y-m), \frac{1}{2}(u y+m))$ to give
\begin{equation}\label{RSeq}
\begin{aligned}
R^2 & = ((x^2-z^2)u^2 + x^2z^2)m^2 - 2u x^2z^2 m y + x^2z^2(u^2-x^2+z^2)y^2, \\ 
	S^2 & = ((x^2-z^2)u^2 + x^2z^2)m^2 + 2u x^2z^2 m y + x^2z^2(u^2-x^2+z^2)
y^2.
\end{aligned}
\end{equation}
This intersection in projective $m,y,R,S$-space is an elliptic curve, since it contains the point
\[ (m,y,R,S) = (x z, \; u, \; x z(x z-u^2), \; x z(x z+u^2)), \]
and a cubic model is 
\begin{align*}
	E: \;Y^2 =  X(X^2 + & 2( (x^2-z^2)u^4 - (x^4-4x^2z^2+z^4)u^2 -x^2z^2(x^2
-z^2)) X \\
	& + (u^2-x^2)^2 (x^2-z^2)^2 (u^2+z^2)^2).
\end{align*}
The curve $E$ has at least two independent points of infinite order,
\begin{align*}
	Q_1 & = ((x^2-u^2)^2z^2, \; (u^2-x^2)^2(u^2x^2+z^4)z), \\
	Q_2 & = ((x^2-z^2)^2u^2, \; (x^2-z^2)^2(u^4+x^2z^2)u).
\end{align*}
We can now pull back points $i_1 Q_1+i_2 Q_2$ to give parametrizations for $(a_0
,a_1,a_2,a_3)$ in terms of $u,x,z$. For example, $Q_2$ pulls back to
\[ (a_0,a_1,a_2,a_3) = (u^2(x+z)^2, \; -4u^2x z, \; 0, \; (u^2-x^2)(u^2-z^2)); \]
and $Q_1$ pulls back to
\begin{align*}
	&	(a_0,a_1,a_2,a_3) = ( \\
	& 	(2 x^2 z^2 u^6 + x^2(x^3-x^2z-2x z^2-6z^3) u^5 -2 z(x^5-x^4z-3x^
3z^2-5x^2z^3+z^5) u^4 \\
	& + 2 z^2(x+z)(x^4-2x^3z-2x^2z^2-2x z^3+z^4) u^3 - 2z^3(x^5-5x^3z^2-3x^2
z^3 \\
& -x z^4+z^5) u^2 - z^6(6x^3+2x^2z+x z^2-z^3) u + 2x^3z^7)^2, \\
	&  -4u x z(u^2 x^2 - 2u x^2 z + 2x z^3 - z^4)(u^2 x^2 - 2u^2 xz + 2u 
z^3 - z^4) \times \\
	& (-u^2 x^2 + 2u^3 z - 2u^2 z^2 + 2u z^3 - z^4)(u^3 x^2 - 2u^2 x^2 z + 2
u x^2 z^2 - 2x^2 z^3 + u z^4), \\
	&  4(u - z)z(x^2 - z^2)(-x^2 + u z)(u^2 x^2 - z^4) (u^2 - u z + z^2) \times \\
	& (u^3 x - 2u^2 x z + u x^2 z - u^2 z^2 + 2u x z^2 - x z^3)(u^3 x - 2u^2
 x z - u x^2 z + u^2 z^2 + 2u x z^2 - x z^3), \\
	&  (u^2 - x^2)(u - z)(u^2 x^2 - 2u x^2 z + 2u z^3 - z^4)(u^2 x^2 - 2u^2 
x z + 2u x z^2 - 2x z^3 + z^4) \times \\
	& (u^2 x^2 + 2u^2 x z - 2u x z^2 + 2x z^3 + z^4)(u^3 x^2 - u^2 x^2 z + 2
u x^2 z^2 - 2u^2 z^3 + u z^4 - z^5)).
\end{align*}
The curve $E$ has a torsion point 
\[ Q_0 = ( -(u^2-x^2)(x^2-z^2)(u^2+z^2), \; 2u x z(x^2-z^2)(u^2-x^2)(u^2+z^2)) \]
of order 4, but the pullbacks of $i_0 Q_0+i_1 Q_1+i_2 Q_2$ give symmetries of the pullbacks with $i_0=0$, and so it is not necessary to consider the point $Q_0$.
\bigskip 
\noindent
We note that the curve (\ref{RSeq}) is singular when $u^2-x^2+z^2=0$, in which case it reduces to the genus zero curve 
\begin{equation}\label{RSeq2}
\begin{aligned}
    R^2 & = m \, ((u^4 + x^2z^2)m - 2u x^2z^2 y), \\
	S^2 & = m \, ((u^4 + x^2z^2)m + 2u x^2z^2 y).
\end{aligned}
\end{equation}
Setting $(u,x,z)=(c^2-d^2,c^2+d^2,2c d)$, we obtain
\begin{align*}
R^2 & = m( (c^8+14c^4d^4+d^8)m - 8c^2d^2(c^2-d^2)(c^2+d^2)^2y),\\
S^2 & = m( (c^8+14c^4d^4+d^8)m + 8c^2d^2(c^2-d^2)(c^2+d^2)^2y).
\end{align*}
From
\begin{align*}
& ( (c^8+14c^4d^4+d^8)m - 8c^2d^2(c^2-d^2)(c^2+d^2)^2y) \; \times \\
& ( (c^8+14c^4d^4+d^8)m + 8c^2d^2(c^2-d^2)(c^2+d^2)^2y)=\square
\end{align*}
we obtain the parameterization 
\[ (m,y) = ( \frac{-A^2-B^2}{c^8+14c^4d^4+d^8}, \; \frac{A B}{4c^2d^2(c^2-d^2)
(c^2+d^2)^2} ), \]
with 
\[ (c^8+14c^4d^4+d^8) R^2 = (A+B)^2 (A^2+B^2), \quad (c^8+14c^4d^4+d^8)S^2 = (A-
B)^2 (A^2+B^2). \]
The latter are parameterized by
\begin{align*}
	(A,B,R,S) = ( -(c^4-d^4) G^2 - 8c^2d^2 G H + (c^4-d^4) H^2, & \\
    4 c^2d^2 G^2 - 2(c^4-d^4) G H - 4c^2d^2 H^2, & \\
	-(G^2 + H^2)((c^4-4c^2d^2-d^4)G^2 + 2(c^4+4c^2d^2-d^4) G H  - (c^4-4c^2d
^2-d^4) H^2), & \\
	(G^2 + H^2)((c^4+4c^2d^2-d^4)G^2 - 2(c^4-4c^2d^2-d^4) G H - (c^4+4c^2d^2
-d^4) H^2) & ).
\end{align*}
Then
\begin{align*}
	& (P,Q) = (\frac{u y-m}{2}, \; \frac{u y+m}{2}) = 1/(4c^2d^2(c^2+d^2)^2)
 \; \times \\
	& ( (4c^2d^4(c^2+d^2)G^4 + (c^8-18c^4d^4+d^8) G^3 H + 8c^2d^2(2c^2-d^2)(
c^2+d^2) G^2 H^2 \\
	& - (c^8-18 c^4d^4+d^8) G H^3 + 4c^2d^4(c^2+d^2) H^4), \\
	& (-4c^4d^2(c^2+d^2)G^4 + (c^8-18c^4d^4+d^8) G^3 H + 8c^2d^2(c^2-2d^2)(c
^2+d^2) G^2 H^2 \\
	& -(c^8-18c^4d^4+d^8) G H^3 -4 c^4d^2(c^2+d^2) H^4)).
\end{align*}
This gives
\[ t=(Q^2-P^2)/(x y z) = -(c^2-d^2)(G^2+H^2)^2/(2c d(c^2+d^2)) \]
from which $(p,q,r,s)=\left( \frac{t z-x y}{2}, \frac{t z+x y}{2}, \frac{y z-t x
}{2}, \frac{y z+t x}{2} \right)$
can be computed, leading to
\begin{equation}\label{a1a2a3a4solution}
\begin{aligned}
	& a_0 \; : \; a_1 \; : \; a_2 \; : \; a_3 = p^2 \; : \; q^2-p^2 \; : \; 
r^2-p^2 \; : \; P^2-p^2 = \\
	& \\
& (c^2+d^2)^2 (-4c^2d^4(c^2-d^2) G^4 + (c^8-18c^4d^4+d^8) G^3 H \\
& + 8c^2d^2(c^2-d^2)(2c^2+d^2) G^2 H^2 - (c^8-18c^4d^4+d^8) G H^3 - 4c^2d^4(c^2-d^2) H^4)^2: \\
	& \\
& 8c^2d^2(c^4-d^4)^2 (G^2+H^2)^2 (2c^2d^2 G^2 - (c^4-d^4) G H - 2c^2d^2 H^2)((c^
4-d^4) G^2 \\
& + 8c^2d^2 G H - (c^4-d^4)H^2): \\
	& \\
& (c^2-d^2)^2(d(c+d) G - c(c-d) H)((d(c-d) G + c(c+d) H) \times \\
	& ((c^2-2c d-d^2) G + (c^2+d^2) H) ((c^2+2c d-d^2) G + (c^2+d^2) H) \\
	& ((c^2+d^2) G - (c^2+2c d-d^2) H) (c(c+d) G - d(c-d) H) \times\\
	& (c(c-d) G + d(c+d) H)((c^2+d^2) G - (c^2-2 c d-d^2) H)): \\
	& \\
& 4c^2d^2G H(G^2-H^2)((c^4-d^4) G + 4c^2d^2 H)(4c^2d^2 G - (c^4-d^4) H) \times \\
	& ((c^4-4c^2d^2-d^4) G + (c^4+4c^2d^2-d^4) H) ((c^4+4c^2d^2-d^4) G - (c^
4-4c^2d^2-d^4) H).
\end{aligned}
\end{equation}
Viewed as polynomials in $\mathbb Z[c,d,G,H]$, the quantities $a_0,a_1,a_2,a_3$ are homogeneous of degree $20$ in $c,d$ and of degree $8$ in $G,H$.

We are now in the position to present the lower bound for $H_{3}(N)$.

\begin{proof}[Proof of Theorem~\ref{thm:H3}]
Taking $(c,d)=(3,1)$ in \eqref{a1a2a3a4solution} and applying
$\phi_1$, followed by the permutation $\phi_\sigma$ interchanging
$a_1$ and $a_3$, gives the following homogeneous binary forms:
\begin{align*}
F_0(G,H)
 &=
25\bigl(
18G^4-319G^3H-684G^2H^2+319GH^3+18H^4
\bigr)^2,\\
F_1(G,H)
 &=
9GH(G-H)(G+H)(9G-20H)(11G+29H)\\
 &\hspace{35mm}\times(20G+9H)(29G-11H),\\
F_2(G,H)
 &=
16(G+5H)(G+6H)(2G-3H)(3G+2H)\\
 &\hspace{25mm}\times(5G-7H)(5G-H)(6G-H)(7G+5H),\\
F_3(G,H)
 &=
7200(G^2+H^2)^2
(9G^2-40GH-9H^2)
(10G^2+9GH-10H^2).
\end{align*}
All four forms are homogeneous of degree $8$. A direct calculation shows that whenever 
\[
7<G/H<8,
\]
all four entries are positive and
\[
F_1(G,H)<F_2(G,H)<F_3(G,H).
\]
Moreover, by construction, $H(F_0;F_1,F_2,F_3)\subset\cal S$.

For $B\geq1$, consider the set of primitive parameter pairs
\[
\mathcal P(B)=
\left\{
(G,H)\in\N^2:
1\leq H\leq B,\quad
7H<G<8H,\quad
\gcd(G,H)=1
\right\}.
\]
For a fixed $H$, writing $G=7H+k$ shows that the number of
admissible values of $G$ is $\varphi(H)$. Consequently,
$$
\#\cal{P}(B) = \sum_{H\leq B}\varphi(H)+O(1)
\gg B^2,
$$
where $\varphi$ is Euler's totient function. 
For $(G,H)\in\mathcal P(B)$, let
\[
D(G,H)=\gcd\bigl(
F_0(G,H),F_1(G,H),F_2(G,H),F_3(G,H)
\bigr).
\]
It is clear that this is also the greatest common divisor of the eight elements of the corresponding Hilbert cube. 
Since all eight elements are squares, their greatest common divisor
is itself a square. Thus
\[
D(G,H)=\delta(G,H)^2
\]
for some positive integer $\delta(G,H)$.

It follows that
\[
b_i(G,H)=\frac{F_i(G,H)}{D(G,H)},
\qquad i=0,1,2,3,
\]
give a reduced Hilbert-cube presentation
\[
\bigl(b_0(G,H), b_1(G,H),b_2(G,H),b_3(G,H)\bigr)
\]
whose associated Hilbert cube is contained in $\cal S$.

We next show that the map from $\mathcal P(B)$ to the resulting
reduced presentations has bounded fibers. Suppose that a reduced
tuple
\[
(b_0,b_1,b_2,b_3)
\]
arises from a pair $(G,H)\in\mathcal P(B)$. Then
\[
\frac{F_1(G,H)}{F_0(G,H)}=\frac{b_1}{b_0},
\]
and hence
\[
b_0F_1(G,H)-b_1F_0(G,H)=0.
\]
Since $F_0$ and $F_1$ are not proportional, the binary form
\[
b_0F_1(X,Y)-b_1F_0(X,Y)
\]
is nonzero and has degree at most $8$. It therefore has at most
$8$ zeros in the projective line. Since each rational projective
point in the region under consideration has a unique primitive
representative with $H>0$, each reduced presentation arises from
at most $8$ pairs in $\mathcal P(B)$.

Finally, since $7< G/H< 8$, there is a constant $C>0$ such
that
$$
\max_{0\leq i\leq3}F_i(G,H)\leq CB^8
$$
for every $(G,H)\in\mathcal P(B)$. Division by $D(G,H)$ can only
decrease the entries. Taking
$$
B=\left\lfloor\left(\frac{N}{C}\right)^{1/8}\right\rfloor
$$
therefore gives
$$H_3(N) \geq \frac{1}{8}\#\mathcal P(B) \gg B^2 \gg N^{1/4}.
$$
This proves the theorem.
\end{proof}

\subsection{Hilbert-cube presentations with two equal generators}
\bigskip

The examples in Table~\ref{smallcubes} show that there are also
Hilbert-cube presentations satisfying $a_1=a_2$. We prove that there are infinitely many such examples. The condition $H(a_{0}; a_{1}, a_{2}, a_{3})\subset\cal{S}$ with $a_{1}=a_{2}$ is equivalent to the solvability of the system
\begin{center}
\begin{tabular}{rrr}
        $a_0=p^2$, & $a_0+a_1=q^2$, & $a_0+2a_1=r^2$, \\
        $a_0+a_3=s^2$, & $a_0+a_1+a_3=t^2$, & $a_0+2a_1+a_3=u^2$,
\end{tabular}
\end{center}
so that $(a_0,a_1,a_3) = (p^2, q^2-p^2, s^2-p^2)$ with
\[ p^2-2q^2+r^2 = 0, \; p^2-q^2-s^2+t^2 = 0, \; 2p^2-2q^2-s^2+u^2 = 0. \]
Set
\[ (p,q,s,t) = (a d-b c, \; a c-b d, \; a d+b c, \; a c+b d) \]
to get
\begin{align*}
(2c^2-d^2) a^2 - 2c d a b - (c^2-2d^2) b^2 & = r^2, \\
(2c^2-d^2) a^2 + 2c d a b - (c^2-2d^2) b^2 & = u^2.
\end{align*}

These two equations define a genus one curve over $\mathbb Q(c,d)$. Since it contains the point at $(a,b,r,u)=(1,1,c-d,c+d)$, it is an elliptic curve.
A cubic model is
\[ Y^2 = X ( X^2 - 4(c^2-c d-d^2)(c^2+c d-d^2) X + 4(c^2-d^2)^4), \]
with points
\[ P_0(X,Y)=((c^2-d^2)^2, \; (c^2-d^2)^2(c^2+d^2)), \quad P_1(X,Y)=(2(c^2-d^2)^2, \; 4c d(c^2-d^2)^2) \]
of infinite order and order 4, respectively.
The pullbacks of multiples of $P_0$ give solutions $(a_0,a_1,a_2,a_3)$ for which $a_1=a_2$. 

To obtain an explicit parametric family, consider the point $2P_0$. Pulling $2P_0$ back through the above transformations gives
\begin{align*}
        (a,b,r,u) = ( & (5c^2-d^2)(c^2+7d^2), \;
        -(c^2-5d^2)(7c^2+d^2), \\
        & (c-d)(c^4+36c^3d+38c^2d^2+36cd^3+d^4), \\
        & (c+d)(c^4-36c^3d+38c^2d^2-36cd^3+d^4))
\end{align*}
and hence
\begin{equation}\label{a1a2solution}
\begin{aligned}
        a_0 &=(c-d)^2
        (7c^4+12c^3d-22c^2d^2+12cd^3+7d^4)^2, \\
        a_1 &=-24(c-d)^2(c+d)^2(c^2+d^2)
        (c^2-6cd+d^2)(c^2+6cd+d^2), \\
        a_3 &=-4cd(c^2-5d^2)(5c^2-d^2)
        (7c^2+d^2)(c^2+7d^2).
\end{aligned}
\end{equation}

The parametrization (\ref{a1a2solution}) allows us to prove Theorem~\ref{thm:idenical-base}.

\begin{proof}[Proof of Theorem \ref{thm:idenical-base}]
We argue as in the proof of Theorem~\ref{thm:H3}, keeping both
homogeneous parameters and counting primitive lattice points in
suitable cones. We only record the features of the present
parametrization that differ from that argument.

Let $A_0(c,d)$, $A_1(c,d)$, and $A_3(c,d)$ denote the three
homogeneous forms on the right-hand side of
\eqref{a1a2solution}, and put $A_2=A_1$. Thus
\[
H(A_0(c,d);A_1(c,d),A_1(c,d),A_3(c,d))
\subset\mathcal S.
\]
All three forms are homogeneous of degree $10$.

Set $\alpha=\frac{3+\sqrt{17}}{4}$. For $1<x<\sqrt{5}$, the values
$ A_0(x,1), A_1(x,1), A_3(x,1)$ are positive. Moreover, a direct factorization shows that
$$
\op{sgn}\bigl(A_3(x,1)-A_1(x,1)\bigr)
=
\op{sgn}(\alpha-x)
$$
throughout this interval. Consequently,
$$
0<A_1(x,1)<A_3(x,1)
\qquad\text{for }1<x<\alpha,
$$
whereas
$$
0<A_3(x,1)<A_1(x,1)
\qquad\text{for }\alpha<x<\sqrt{5}.
$$

We now apply the counting argument from the proof of
Theorem~\ref{thm:H3}. For either of the fixed intervals
$$
I_-=(1,\alpha)
\qquad\text{and}\qquad
I_+=(\alpha,\sqrt{5}),
$$
consider the primitive parameter pairs
$$
\mathcal P_I(B)= 
\left\{
(c,d)\in\mathbb N^2:
d\leq B,\quad \frac{c}{d}\in I,\quad \gcd(c,d)=1
\right\}.
$$
It is clear that $\#\mathcal P_I(B)\gg B^2$.

As in the proof of Theorem~\ref{thm:H3}, the common divisor of
the entries of each resulting presentation is a square, so
division by this common divisor produces a reduced presentation
and preserves all the square conditions. The map from primitive
parameter pairs to reduced presentations has uniformly bounded
fibers. Indeed, for any fixed reduced image, the value of
\[
\frac{A_1(c,d)}{A_0(c,d)}
\]
is fixed, and hence its preimages satisfy a nonzero homogeneous
equation of degree $10$. Thus every reduced presentation has at
most $10$ primitive preimages.

Since the forms $A_0,A_1,A_3$ are homogeneous of degree $10$,
their values are $O(B^{10})$ on either of the above cones.
Taking
\[
B\asymp N^{1/10}
\]
therefore yields
\[
\gg B^2\gg N^{1/5}
\]
distinct reduced presentations whose entries are at most $N$.

For $(c,d)\in\mathcal P_{I_-}(B)$, these presentations satisfy
\[
a_1=a_2<a_3.
\]
For $(c,d)\in\mathcal P_{I_+}(B)$, we have
$A_3(c,d)<A_1(c,d)$, and permuting the generators gives
presentations satisfying
\[
a_1<a_2=a_3.
\]
This proves both assertions.
\end{proof}

\begin{rem}
{\rm We can also find reduced Hilbert-cube presentations $(a'_{0}, a'_{1},a'_{2},a'_{3})$ where $a'_2=a'_3$. 
Indeed, it suffices to determine conditions on $c,d$ in~ (\ref{a1a2solution}) such that $a_{i}(c,d)>0$ and $a_{2}(c,d)>a_{3}(c,d)$. Then $(a_{0}, a_{3},a_{1},a_{2})$ is the Hilbert-cube presentation we are looking for. By homogeneity of the polynomials $a_{i}(c,d)$, we may assume that $d=1$ and seek rational values of $c$ satisfying the corresponding inequalities.
One can easily check that if 
$$
c \in\left(\frac{1}{4} \left(3+\sqrt{17}\right), \sqrt{5} \right)=I,
$$
then $a_{i}(c,1)>0$ and $a_{2}(c,1)>a_{3}(c,1)$. Since $I$ contains infinitely many rational numbers, we obtain infinitely many Hilbert-cube presentations satisfying the required condition.
}    
\end{rem}
\bigskip

\section{Proof of Corollary~\ref{cor-fractions} }

We apply our parametric families of Hilbert-cube presentations whose associated cubes are contained in the set of squares 
to prove the corollary. We restate it for convenience:

Let $\cal{H}_{+}=\cal{H}\cap \N^{4}$. For each $i, j\in\{0, 1, 2, 3\}$ with $i<j$, the set
$$
H_{i,j}=\left\{\frac{a_{i}}{a_{j}}:\;(a_{0}, a_{1}, a_{2}, a_{3})\in\cal{H}_{+}\right\}\subset \Q_{+}
$$
is dense in the Euclidean topology in the set $\R_{+}$.

\begin{proof}
The arguments for the six pairs $i<j$ are similar; we treat them in turn.
We first show $\overline{H_{0,1}}=\R_{+}$. Define the function
\[ f_{0,1}(x) = \frac{a_0(x,1)}{a_1(x,1)} = \frac{ -(7+12x-22x^2+12x^3+7x^4)^2} {24(1 + x)^2(1 + x^2)(1-6x+x^2)(1+6x+x^2)} \]
where $a_0$, $a_1$ are the polynomials given by (\ref{a1a2solution}). Note that $f_{0,1}(c/d) = a_0(c,d)/a_1(c,d)$. Figure \ref{fig:Pic_f} presents the graph of the function $f_{0,1}(x)$ for $x\in [-3-2\sqrt{2}, 2]$. 

\begin{figure}[!ht]
    \centering
    \includegraphics[width=0.8\linewidth]{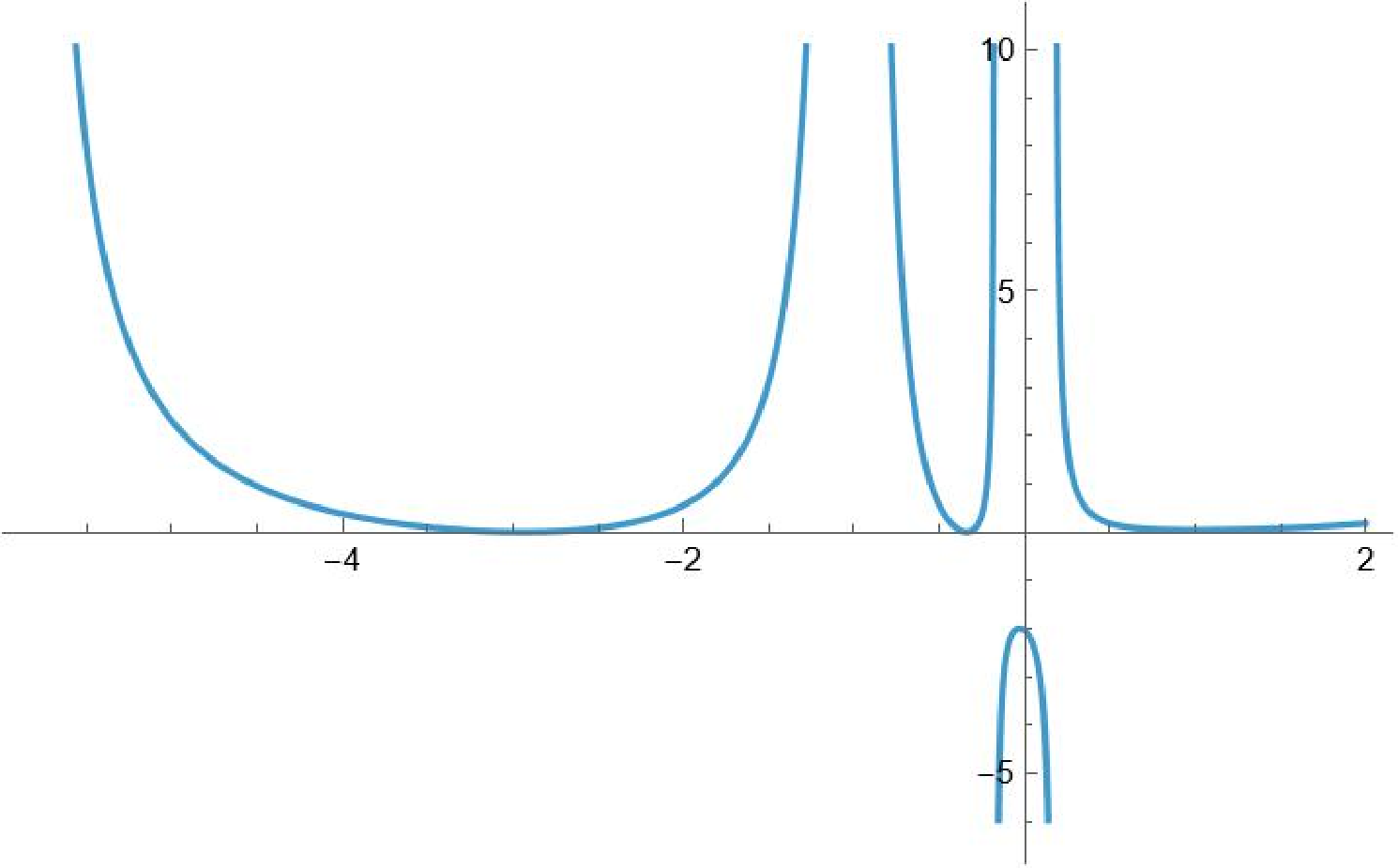}
    \caption{Plot of the function $f_{0,1}(x), x\in [-3-2\sqrt{2}, 2]$.}
    \label{fig:Pic_f}
\end{figure}

The functions $a_i(x,1)$ for $i=1,2,3$ are all positive on the intervals
\[ x \in (-3-2\sqrt{2}, \; -\sqrt{5}) \; \cup \; \left(-\frac{1}{\sqrt{5}}, \; -\frac{3+\sqrt{17}}{4}\right) \; \cup \; \left(\frac{1}{\sqrt{5}}, \; \sqrt{5}\right), \].
Since $a_0(x,1)$ is a square, this implies that $f_{0,1}(x)$ is non-negative on these ranges.
The function $f_{0,1}(x)$ has a local minimum value of $0$ on the interval $(-3-2\sqrt{2}, \; -\sqrt{5})$. This minimum is attained at $x_{min} = \frac{(-3-2\sqrt[4]{2})}{2\sqrt{2}-1} \sim -2.94155$. Define $I_{0,1} = (-3-2\sqrt{2}, \; x_{min})$.
We have
\[ \lim_{x \rightarrow (-3-2\sqrt{2})^{+}} f_{0,1}(x) = +\infty, \]
and $f_{0,1}(x)$ is continuous and decreasing on $(-3-2\sqrt{2}, \; x_{min})$.
It follows that
\[ \{ f_{0,1}(x) : x \in \Q \cap I_{0,1} \} \subset H_{0,1}. \]
The density of $\Q \cap I_{0,1}$ in $I_{0,1}$ implies that the Euclidean closure of the set $H_{0,1}$ is equal to $\R_{+}$.
Since $a_1=a_2$, the set $H_{0,2}$ is also dense in $\R_{+}$. \\ \\
For the density of $H_{0,3}$, we follow the same approach and define
\[ f_{0,3}(x) = \frac{a_0(x,1)}{a_3(x,1)} = \frac{-(x-1)^2 (7+12x-22x^2+12x^3+7x^4)^2}{4x(x^2-5)(5x^2-1)(x^2+7)(7x^2+1)}. \]

Let $I_{0,3}$ be the interval $(1, \sqrt{5})$, on which $a_i(x,1)$, $i=0,\ldots,3$, are all positive. Then
\[ \{ f_{0,3}(x): x \in \Q_{+} \cap I_{0,3} \} \subset H_{0,3}. \]
Since $f_{0,3}$ is continuous and increasing on $I_{0,3}$, with $f_{0,3}(1)=0$ and
\[ \lim_{x \rightarrow \sqrt{5}^{-}} f_{0,3}(x) = +\infty, \]
it follows that $H_{0,3}$ is dense in $\R_{+}$. \\ \\
To prove the density of $H_{1,3}$, consider the function
\[ f_{1,3}(x) = \frac{a_1(x,1)}{a_3(x,1)} = \frac{ 6(1-x)^2(1+x)^2(1+x^2)(1-6x+x^2)(1+6x+x^2)}{ x(5-x^2)(1-5x^2)(7+x^2)(1+7x^2)}. \]
Let $I_{1,3} = (1, \sqrt{5})$, on which $a_i(x,1)$, $i=0,...,3$, are all positive. Then
\[ \{ f_{1,3}(x): x \in \Q_{+} \cap I_{1,3} \} \subset H_{1,3}. \]
Since $f_{1,3}$ is continuous and increasing on $I_{1,3}$, with $f_{1,3}(1)=0$ and
\[ \lim_{x \rightarrow \sqrt{5}^{-}} f_{1,3}(x) = +\infty, \]
it follows that $H_{1,3}$ is dense in $\R_{+}$. By applying the same function, we deduce also that $H_{2,3}$ is dense in $\R_{+}$. \\ \\
Finally, to prove the density of $H_{1,2}$, we cannot, of course, use the polynomials (\ref{a1a2solution}), where $a_1=a_2$. Instead, we return to (\ref{a1a2a3a4solution}) and consider, for example, the function
\begin{align*}
  f_{1,2}(x) & = \frac{a_1(x,1,2,1)}{a_2(x,1,2,1)} \\
  & = \frac{400x^2(1+x^2)^2(-1-3x^2+x^4)(-3+16x^2+3x^4)}{(9+2x^2+x^4)(4-13x^2+x^4)(1-22x^2+9x^4)(1+3x^2+4x^4)}.
\end{align*}
Let $I_{1,2}=\left( \frac{2+\sqrt{7}}{3}, \sqrt{ \frac{3+\sqrt{13}}{2}} \right)$, on which $a_i(x,1,2,1)$, $i=0,1,2,3$, are all positive. Then
\[ \{ f_{1,2}(x): x \in \Q_{+} \cap I_{1,2} \} \subset H_{1,2}. \]
Since $f_{1,2}$ is continuous and decreasing on $I_{1,2}$, satisfies 
$f_{1,2}\left( \sqrt{ \frac{3+\sqrt{13}}{2}} \right)=0$, and
\[ \lim_{x \rightarrow \left( \frac{2+\sqrt{7}}{3} \right)^{+}} f_{1,2}(x) = +\infty, \]
it follows that $H_{1,2}$ is dense in $\R_{+}$.

\end{proof}

\section{Hilbert-cube presentations with the same values of $a_0$, $a_1$, $a_2$}
We have observed numerically many pairs of Hilbert-cube presentations $(a_0, a_1,a_2,a_3)$, $(A_0, A_1,A_2,A_3)$ such that $a_i=A_i$, $i=0,1,2$. More precisely, in the range $a_{0}<a_{1}<10^8$ we found exactly six such pairs; they are listed in Table~\ref{tableaiAi}.
These six pairs actually represent twelve Hilbert cubes. In each pair the first three parameters are identical. Additionally, for many of these we have $a_1=a_2$, and the value
71831760 occurs twice in 4 of these 12 Hilbert cubes.

\begin{table}[htbp]
\centering
\begin{tabular}{|l|l|l|l|l|}
\hline
$a_{0}=A_{0}$ & $a_{1}=A_{1}$ & $a_{2}=A_{2}$ & $a_{3}$ & $A_{3}$ \\
\hline
332929  & 6726720  & 6726720  & 8322435   & 22381827 \\
438244  & 1004157  & 1939520  & 3013920   & 8791200  \\
643204  & 1367520  & 1367520  & 6804237   & 35947197 \\
4674244 & 1367520  & 1367520  & 2773197   & 31916157 \\
4713241 & 71831760 & 71831760 & 130613448 & 665527080 \\
38775529 & 71831760 & 71831760 & 96551160  & 631464792 \\
\hline
\end{tabular}
\caption{ Hilbert-cube presentations $(a_0, a_1, a_2, a_3)$ and $(A_0, A_1, A_2, A_3)$ whose associated cubes are contained in $\cal{S}$ with $a_{i}=A_{i}$ for $i=0, 1, 2$.}
\label{tableaiAi}
\end{table}
The presented data suggest that there are infinitely many such pairs. The next theorem confirms this expectation.
\begin{thm}
    There are infinitely many pairs 
    $$
    (a_{0}, a_{1}, a_{2}, a_{3}),\; (A_0, A_1, A_2, A_3)
    $$ 
    of Hilbert-cube presentations whose associated cubes are contained in $\cal S$ and such that $a_{i}=A_{i}$ for $i=0, 1, 2$.
\end{thm}
\begin{proof} 
The defining system is
\[
\begin{alignedat}{4}
a_0 &= p^2, \qquad&
a_0+a_1 &= q^2, \qquad&
a_0+a_2 &= r^2, \qquad&
a_0+a_1+a_2 &= s^2, \\[2mm]
a_0+a_3 &= P_1^2, &
a_0+a_1+a_3 &= Q_1^2, &
a_0+a_2+a_3 &= R_1^2, &
a_0+a_1+a_2+a_3 &= S_1^2, \\[2mm]
a_0+A_3 &= P_2^2, &
a_0+a_1+A_3 &= Q_2^2, &
a_0+a_2+A_3 &= R_2^2, &
a_0+a_1+a_2+A_3 &= S_2^2.
\end{alignedat}
\]
On eliminating $a_0,a_1,a_2,a_3,A_3$, this is equivalent to the system
\[ p^2+s^2=q^2+r^2, \qquad P_1^2+S_1^2=Q_1^2+R_1^2, \qquad P_2^2+S_2^2=Q_2^2+R_2^2, \]
\[ p^2-q^2 = P_1^2-Q_1^2 = P_2^2-Q_2^2, \qquad p^2-r^2 = P_1^2-R_1^2 = P_2^2-R_2^2. \]
To satisfy the first row, set
\begin{align*}
	(p,q,r,s) = & (\alpha_0 \delta_0-\beta_0 \gamma_0, \;\; \alpha_0 \delta_0+\beta_0 \gamma_0, \;\; \alpha_0 \gamma_0-\beta_0 \delta_0, \;\; \alpha_0 \gamma_0+\beta_0 \delta_0), \\
	(P_1,Q_1,R_1,S_1) = & (\alpha_1 \delta_1-\beta_1 \gamma_1, \;\; \alpha_1
 \delta_1+\beta_1 \gamma_1, \;\; \alpha_1 \gamma_1-\beta_1 \delta_1, \;\; \alpha_1 \gamma_1+\beta_1 \delta_1), \\
	(P_2,Q_2,R_2,S_2) = & (\alpha_2 \delta_2-\beta_2 \gamma_2, \;\; \alpha_2
 \delta_2+\beta_2 \gamma_2, \;\; \alpha_2 \gamma_2-\beta_2 \delta_2, \;\; \alpha_2 \gamma_2+\beta_2 \delta_2).
\end{align*}
The second row then yields
\[ \alpha_0 \beta_0 \gamma_0 \delta_0 = \alpha_1 \beta_1 \gamma_1 \delta_1 = \alpha_2 \beta_2 \gamma_2 \delta_2, \]
\[ (\alpha_0^2-\beta_0^2)(\gamma_0^2-\delta_0^2) = (\alpha_1^2-\beta_1^2)(\gamma
_1^2-\delta_1^2) = (\alpha_2^2-\beta_2^2)(\gamma_2^2-\delta_2^2). \] 
If we set
\[ (\alpha_i,\beta_i)=\left( \frac{\gamma_i+\delta_i}{2},\;\frac{\gamma_i-\delta_i}{2} \right), \quad i=0,1,2, \]
then we obtain the simple system
\[ \gamma_0 \delta_0(\gamma_0^2-\delta_0^2) = \gamma_1 \delta_1(\gamma_1^2-\delta_1^2) = \gamma_2 \delta_2(\gamma_2^2-\delta_2^2). \]
Thus we require three Pythagorean triangles with equal areas. Such triangles are provided by
\begin{align*}
	(\gamma_0,\delta_0) = & (u^2+u v+v^2, u^2-v^2), \\
	(\gamma_1,\delta_1) = & (u^2+u v+v^2, 2u v+v^2), \\
	(\gamma_2,\delta_2) = & (u^2+2 u v, u^2+u v+v^2),
\end{align*}
with 
\begin{align*}
	(\alpha_0,\beta_0) = & (u(2u+v)/2, v(u+2v)/2), \\
	(\alpha_1,\beta_1) = & ((u+v)(u+2v)/2, u(u-v)/2), \\
	(\alpha_2,\beta_2) = & ((u+v)(2u+v)/2, (u-v)v/2).
\end{align*}
Then
\begin{align*}
	(p,q,r,s) = & ((2u^4 - 5u^2v^2 - 4uv^3 - 2v^4)/2, \\
	& (2u^4 + 2u^3v + u^2v^2 + 2uv^3 + 2v^4)/2, \\
	& (2u^4 + 2u^3v + u^2v^2 + 2uv^3 + 2v^4)/2, \\
	& (2u^4 + 4u^3v + 5u^2v^2 - 2v^4)/2);
\end{align*}
\begin{align*}
	(P_1,Q_1,R_1,S_1) = & ((-u^4 + 2u^3v + 7u^2v^2 + 8uv^3 + 2v^4)/2, \\
	& (u^4 + 2u^3v + 7u^2v^2 + 6uv^3 + 2v^4)/2,  \\
	& (u^4 + 2u^3v + 7u^2v^2 + 6uv^3 + 2v^4)/2, \\
	& (u^4 + 6u^3v + 5u^2v^2 + 4uv^3 + 2v^4)/2)
\end{align*}
\begin{align*}
	(P_2,Q_2,R_2,S_2) = & ((2u^4 + 4u^3v + 5u^2v^2 + 6uv^3 + v^4)/2, \\
	& (2u^4 + 6u^3v + 7u^2v^2 + 2uv^3 + v^4)/2, \\
	& (2u^4 + 6u^3v + 7u^2v^2 + 2uv^3 + v^4)/2, \\
	& (2u^4 + 8u^3v + 7u^2v^2 + 2uv^3 - v^4)/2).
\end{align*}
Finally, scaling by a factor 4, 
\begin{align*}
	(a_0,a_1,a_2,a_3) = &  ( (2u^4 - 5u^2v^2 - 4uv^3 - 2v^4)^2, \\
	& 4u(u - v)v(u + v)(2u + v)(u + 2v)(u^2 + uv + v^2), \\
	& 4u(u - v)v(u + v)(2u + v)(u + 2v)(u^2 + uv + v^2), \\
	& -u(u + 2v)(u^2 - 2uv - 2v^2)(u^2 + 2v^2)(3u^2 + 4uv + 2v^2) );
\end{align*}
\begin{align*}
	(A_0,A_1,A_2,A_3) = & ((2u^4 - 5u^2v^2 - 4uv^3 - 2v^4)^2, \\
	& 4u(u - v)v(u + v)(2u + v)(u + 2v)(u^2 + uv + v^2), \\
	& 4u(u - v)v(u + v)(2u + v)(u + 2v)(u^2 + uv + v^2), \\
	& v(2u + v)(2u^2 + 2uv - v^2)(2u^2 + v^2)(2u^2 + 4uv + 3v^2)),
\end{align*}
where $a_0=A_0$, $a_1=A_1$, $a_2=A_2$.
\end{proof}

\section{Further remarks and computational observations}\label{sec4}

In this section we formulate some conjectures which may stimulate
further research. We have already answered
Problems~\ref{prob:equal-generators} and~\ref{prob:effective}.
We use our computational approach to study
Problem~\ref{prob:fixed-a0}. More precisely, we are interested in whether, for a given $n\in\N$, there is a reduced Hilbert-cube presentation $(a_{0}, a_{1}, a_{2}, a_{3})$ whose associated cube is contained in $\cal{S}$ such that $a_{0}=n^2$. Note that the assumption that the presentation $(a_{0}, a_{1}, a_{2}, a_{3})$ must be reduced is important and makes the question nontrivial. Indeed, without this assumption we get the presentation $(n^2, 528n^2, 840n^2, 840n^2)$ whose associated cube is contained in $\cal{S}$. 

If $a_0=n^2$ is fixed, then necessarily $a_1=t^2-n^2$. Using a small modification of Algorithm~1, we searched for reduced three-dimensional Hilbert-cube presentations with $a_0=n^2$ and $a_1\leq(1.4\cdot10^8)^2$ for $n\in\{1,\ldots,100\}$.  These computations took several days on the personal computer of the third author. In this range we found a suitable solution for all but the four values $n=42, 44, 56, 80$. For some values of $n$ the examples are quite large. We collect our findings in Tables~\ref{tablen1} and~\ref{tablen2}.
\begin{landscape}
\begin{table}[htbp]
\centering
\begin{tabular}{|l|l|l|l||l|l|l|l|}
\hline
$a_{0}$ & $a_{1}$ & $a_{2}$ & $a_{3}$ &
$a_{0}$ & $a_{1}$ & $a_{2}$ & $a_{3}$ \\
\hline
$1^2$  & 528 & 840 & 840 & $26^2$ & 151826342525 & 3489610801824 & 66625590001824 \\
$2^2$  & 3360 & 3360 & 9405 & $27^2$ & 8334040 & 15704640 & 33280632 \\
$3^2$  & 13704795 & 35772352 & 57562560 & $28^2$ & 44782080 & 97041417 & 730728240 \\
$4^2$  & 565488 & 3094065 & 4309760 & $29^2$ & 112728 & 159960 & 1584240 \\
$5^2$  & 32736 & 46200 & 54264 & $30^2$ & 21903703101 & 151484423200 & 356015088000 \\
$6^2$  & 7215822880 & 10763232480 & 26819630253 & $31^2$ & 10920 & 10920 & 26928 \\
$7^2$  & 5280 & 5280 & 9555 & $32^2$ & 647548785 & 5626199040 & 6457085712 \\
$8^2$  & 836559012432 & 77545495104000 & 1563557678865 & $33^2$ & 76396812236640 & 145116788951160 & 579009729388552 \\
$9^2$  & 1951960680 & 16612374240 & 39358195240 & $34^2$ & 3885 & 73920 & 73920 \\
$10^2$ & 2400 & 4389 & 8736 & $35^2$ & 121275 & 315744 & 1154400 \\
$11^2$ & 4844280 & 7134120 & 23746008 & $36^2$ & 300986505 & 660694320 & 776624128 \\
$12^2$ & 257920 & 266112 & 3932145 & $37^2$ & 110590166232 & 193219146120 & 738057963240 \\
$13^2$ & 10440 & 15960 & 62832 & $38^2$ & 804196431456 & 7816246018581 & 1829965027200 \\
$14^2$ & 292485 & 487008 & 2328480 & $39^2$ & 160888 & 1318680 & 8172360 \\
$15^2$ & 157326624 & 14657944675 & 183629376 & $40^2$ & 32367246681 & 161813106000 & 566840338944 \\
$16^2$ & 175305 & 802560 & 929040 & $41^2$ & 68986596728 & 91584309960 & 92354600520 \\
$17^2$ & 45936 & 134400 & 249711 & $42^2$ &   &   &   \\
$18^2$ & 294525 & 655776 & 2855776 & $43^2$ & 839040 & 4362072 & 4462920 \\
$19^2$ & 11898227880 & 11898227880 & 15944365080 & $44^2$ &   &   &   \\
$20^2$ & 639600 & 2246601 & 3489024 & $45^2$ & 538200 & 1065064 & 2952936 \\
$21^2$ & 17412649408 & 246163829760 & 1357746969735 & $46^2$ & 3360 & 3360 & 7293 \\
$22^2$ & 66617760 & 688642080 & 178836645 & $47^2$ & 13167 & 349440 & 526320 \\
$23^2$ & 41496 & 138600 & 138600 & $48^2$ & 13503321 & 58489600 & 98960400 \\
$24^2$ & 166512640 & 184524480 & 6527508273 & $49^2$ & 531960 & 1785168 & 3796200 \\
$25^2$ & 2429856 & 3377619 & 10594400 & $50^2$ & 39525 & 384384 & 1329216 \\
\hline
\end{tabular}
\caption{Hilbert-cube presentations $(a_{0}, a_{1}, a_{2}, a_{3})$ whose associated cubes are contained in $\cal{S}$ with $a_{0}=n^2, n\in\{1,\ldots, 50\}$.}
\label{tablen1}
\end{table}
\end{landscape}

\begin{landscape}
\begin{table}[htbp]
\centering
\begin{tabular}{|l|l|l|l||l|l|l|l|}
\hline
$a_{0}$ & $a_{1}$ & $a_{2}$ & $a_{3}$ &
$a_{0}$ & $a_{1}$ & $a_{2}$ & $a_{3}$ \\
\hline
$51^2$ & 166320 & 166320 & 2120248 & $76^2$ & 549114862234224 & 4122693075650625 & 20196553848570624 \\
$52^2$ & 11300325105 & 31353073920 & 33667843440 & $77^2$ & 12195120 & 74419200 & 887319015 \\
$53^2$ & 21216 & 28875 & 153216 & $78^2$ & 748672960 & 1037735712 & 10501529445 \\
$54^2$ & 9790308000 & 44245016109 & 68817025984 & $79^2$ & 286440 & 286440 & 901968 \\
$55^2$ & 15057741075 & 70139692896 & 92139564000 & $80^2$ &   &   &   \\
$56^2$ &   &   &   & $81^2$ & 6758640 & 28851823 & 76573440 \\
$57^2$ & 708142072 & 7438541760 & 12494094480 & $82^2$ & 31051605 & 182189280 & 241422720 \\
$58^2$ & 44160 & 94605 & 110880 & $83^2$ & 207480 & 337680 & 1416360 \\
$59^2$ & 25080 & 65688 & 93240 & $84^2$ & 1028093040 & 4312147833 & 12967736320 \\
$60^2$ & 255996400 & 1462599936 & 1795637025 & $85^2$ & 421800 & 4232256 & 7905744 \\
$61^2$ & 21983000 & 3335636304 & 3802075200 & $86^2$ & 27387360 & 30319653 & 54073920 \\
$62^2$ & 6765 & 43680 & 43680 & $87^2$ & 27142672471200 & 51905252513331 & 65136117775456 \\
$63^2$ & 7108920 & 16052080 & 33070032 & $88^2$ & 1414053072 & 5534161920 & 6435721985 \\
$64^2$ & 62985 & 270480 & 346368 & $89^2$ & 1899240 & 1899240 & 2299440 \\
$65^2$ & 9488336 & 19505664 & 44618175 & $90^2$ & 8029125 & 21763456 & 126824544 \\
$66^2$ & 3659040 & 4165408 & 4752405 & $91^2$ & 16728000 & 18559200 & 40568619 \\
$67^2$ & 47040 & 86112 & 126555 & $92^2$ & 110851377 & 149026800 & 637655040 \\
$68^2$ & 54933120 & 60259545 & 82514432 & $93^2$ & 200200 & 407376 & 488376 \\
$69^2$ & 51408 & 196840 & 344520 & $94^2$ & 200928 & 1551165 & 3109920 \\
$70^2$ & 36309 & 79200 & 114816 & $95^2$ & 7125883200 & 8042493375 & 213369153536 \\
$71^2$ & 1459311360 & 1730222175 & 3062842608 & $96^2$ & 32875482640 & 121549143105 & 209822618880 \\
$72^2$ & 10320319737 & 14224850640 & 178225815040 & $97^2$ & 1294755 & 2015520 & 12728352 \\
$73^2$ & 29640 & 224112 & 304920 & $98^2$ & 118560 & 908160 & 1456917 \\
$74^2$ & 405405 & 6723360 & 20082848 & $99^2$ & 45674280 & 7131286008 & 929691280 \\
$75^2$ & 1307691 & 4209184 & 7365600 & $100^2$ & 309225 & 354816 & 836400 \\
\hline
\end{tabular}
\caption{Hilbert-cube presentations $(a_{0}, a_{1}, a_{2}, a_{3})$ whose associated cubes are contained in $\cal{S}$ with $a_{0}=n^2, n\in\{51,\ldots, 100\}$.}
\label{tablen2}
\end{table}
\end{landscape}

\begin{conj}
For each $n\in\N$, there is a reduced three-dimensional Hilbert-cube presentation $(a_{0}, a_{1}, a_{2}, a_{3})$ whose associated cube is contained in $\cal{S}$ with $a_{0}=n^2$.     
\end{conj}

\bigskip
In Theorem \ref{thm:H3}, we established the lower bound
\[
H_{3}(N)\gg N^{1/4}.
\]
Figures \ref{fig:Pic4_1} and \ref{fig:Pic4_2} show the behavior of the normalized quantities
\[
\frac{H_{3}(N)}{N^{1/4}}
\qquad\text{and}\qquad
\frac{H_{3}(N)}{N},
\]
respectively, for \(n\leq 10^6\). From these pictures one can conjecture that the following limits hold:
$$
\lim_{N\rightarrow +\infty}\frac{H_{3}(N)}{N^{1/4}}=+\infty,\quad \quad\lim_{N\rightarrow +\infty}\frac{H_{3}(N)}{N}=0.
$$
\begin{figure}[!ht]
    \centering
    \includegraphics[width=1\linewidth]{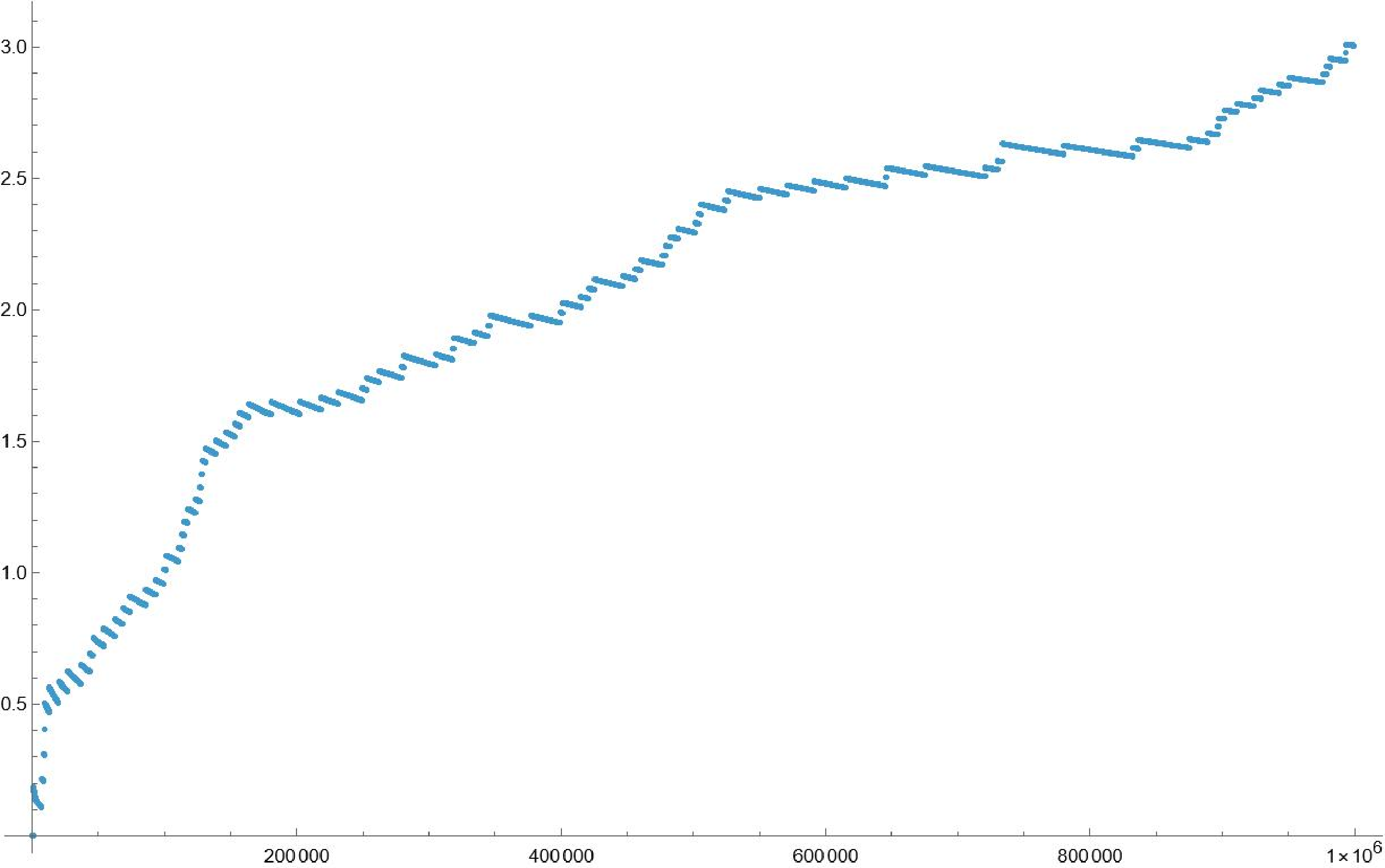}
    \caption{Plot of the function $H_{3}(N)/N^{1/4}$  for $N\leq 10^6$.}
    \label{fig:Pic4_1}
\end{figure}

\begin{figure}[!ht]
    \centering
    \includegraphics[width=1\linewidth]{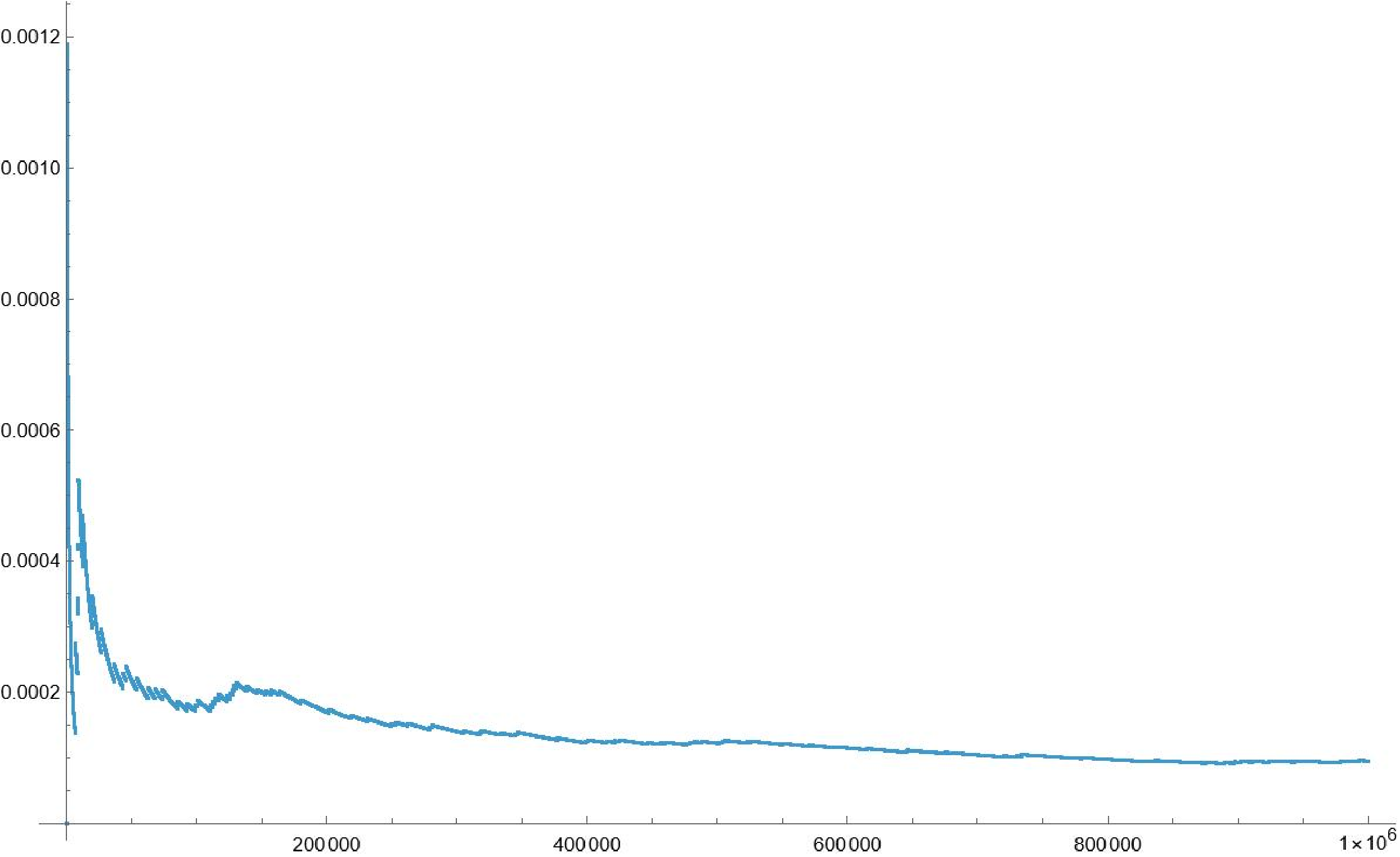}
    \caption{Plot of the function $H_{3}(N)/N$  for $n\leq 10^6$.}
    \label{fig:Pic4_2}
\end{figure}

\begin{prob}
What is the true order of magnitude of the function $H_{3}(N)$?
\end{prob}

To get an idea of the expected behaviour, we used the {\bf FindFit} procedure in Mathematica 14.2 \cite{Wol}. The command
\begin{center}
{\bf FindFit}[data, expression, parameters, variable]
\end{center}
determines the numerical values of the parameters that make the given expression provide the best fit to the data as a function of the specified variable.
In our case, the data set consists of the values of $H_{3}(N)$ for $N \in \{1, \ldots, 10^7\}$, the expression is $F(x)=a x^{b}$, the parameters are $a$ and $b$, and the variable is $x$.
For this setup, the procedure finds that the best-fit parameters are $a = 0.01798344440392967$ and $b = 0.6220626625113858$. The plots of $H_{3}(n)$ and $F(n)$ and the absolute value of their difference are given below. Over the range considered, the function $F(N)$ provides a good approximation to $H_{3}(N)$, since 
$$
\max\{|H_{3}(N)-F(N)|:\;N\leq 10^7\}=15.474.
$$
This experimental approach suggests that it is reasonable to expect that $H_{3}(N)\gg N^{b}$ for some $0.5<b<1$, maybe even $b>0.6$.

\begin{figure}[!ht]
    \centering
    \includegraphics[width=1\linewidth]{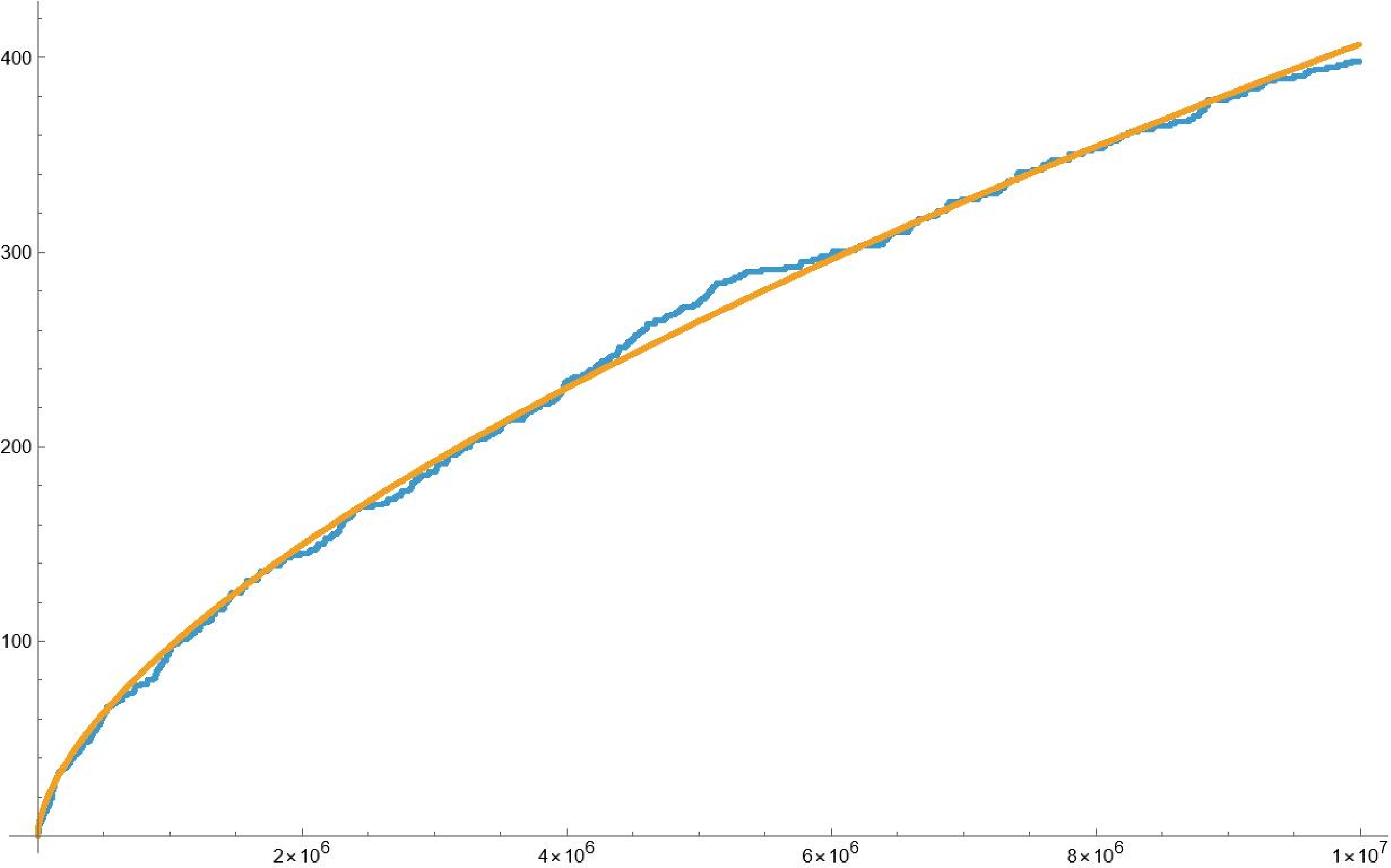}
    \caption{The plot of the functions $H_{3}(N)$ (blue) and the fitted function $F(N)=0.01798N^{0.62206}$ (yellow) for $N\leq 10^7$.}
    \label{fig:Pic4}
\end{figure}

\begin{figure}[!ht]
    \centering
    \includegraphics[width=1\linewidth]{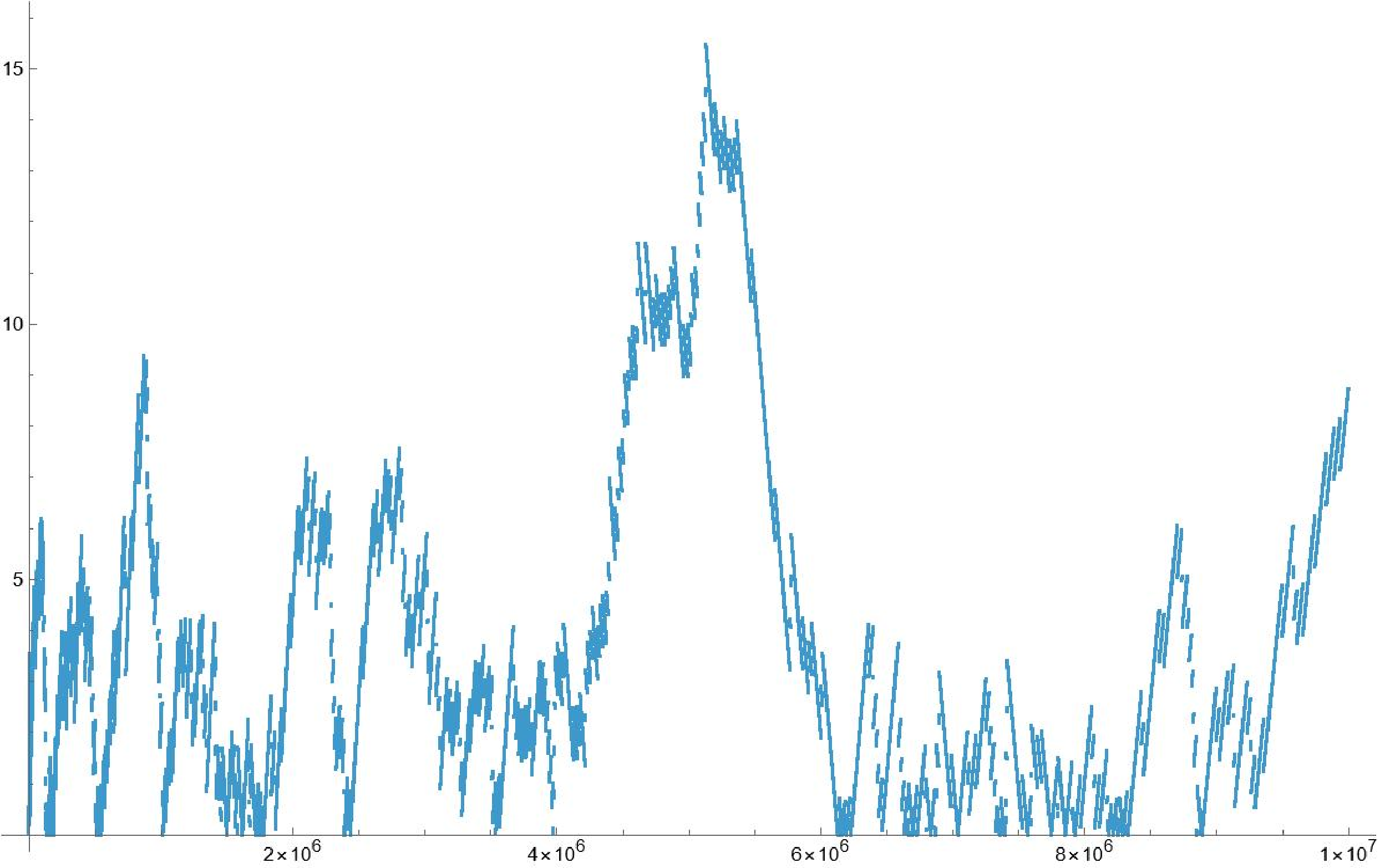}
    \caption{The plot of $|H_{3}(N)-F(N)|$, where $F(N)$ is the fitted function $F(N)=0.01798N^{0.62206}$ for $N\leq 10^7$.}
    \label{fig:Pic4error}
\end{figure}
\bigskip

Having proved that there are infinitely many three-dimensional Hilbert-cube presentations whose associated cubes are contained in the squares, we naturally return to Problem~\ref{prob:d4}.

Although we tried hard, we were unable to find such a four-dimensional Hilbert-cube presentation. However, using the parametrization given by (\ref{a1a2solution}), we were able to find the following:
\begin{align*}
    a_0 = & (c-1)^4(7c^4 + 12c^3 - 22c^2 + 12c + 7)^2, \\
a_1 = & a_2 = -24(c-1)^4(c+1)^2(c^2+1)(c^2 - 6c + 1)(c^2 + 6c + 1), & \\
a_3 = & -4c(c-1)^2(c^2-5)(c^2+7)(5c^2-1)(7c^2+1), \\
a_4 = & -\frac{1}{4}(c^3+c^2+19c-5)(3c^3-c^2+9c+5)(5c^3-19c^2-c-1)(5c^3+9c^2-c+3)
\end{align*}
where thirteen of the sixteen sums are squares (the exceptions being $a_0+a_4$, $a_0+a_1+a_2+a_4$, $a_0+a_1+a_2+a_3+a_4$).

Numerically, we can do a little better. Consider a three-dimensional Hilbert-cube presentation $(a_0, a_1,a_1,a_3)$. Extending it to a four-dimensional Hilbert-cube presentation requires finding $a_4=X$ such that all six quantities
\[ a_0+X, \; a_0+a_1+X, \; a_0+2a_1+X, \; a_0+a_3+X, \; a_0+a_1+a_3+X, \; a_0+2a_1+a_3+X, \]
are squares. Consider, for example, the elliptic curve
\[ Y^2 = (a_0+X)(a_0+a_1+X)(a_0+2a_1+X), \]
with Mordell-Weil group $\mathcal{G}$. Since $X=0$ yields a point of infinite order, the group has rank at least $1$. The points $(X,Y)$ with $a_0+X$, $a_0+a_1+X$, $a_0+2a_1+X$ all being squares are precisely the points of $2\, \mathcal{G}$. This guarantees finding $X$ with three of the above six quantities being squares. Let $X$ range over the $X$-coordinates of $2\mathcal{G}$, and check to see whether any of
\[ a_0+a_3+X, \; a_0+a_1+a_3+X, \; a_0+2a_1+a_3+X  \]
are squares. We implemented this procedure for the parametrized three-dimensional Hilbert-cube presentations in~(\ref{a1a2solution}), for small values of $c$ and $d$.  Unfortunately, we found only a few examples with one additional square. Thus, at present, we know only pseudo-$4$-cubes in which fourteen of the sixteen sums are squares, for example
\[ 6310^2, \; 105386400, \; 105386400, \; 144545984, \; -121397859. \]

\bigskip

\section*{Acknowledgements}

We thank both referees for their careful reading of the manuscript and
for their helpful comments. We are especially grateful to one of the
referees for an exceptionally thorough and constructive report. The
detailed remarks and suggestions helped us clarify several points and
improve the presentation of the paper.

\end{document}